\documentclass[a4paper]{amsart}
\usepackage[english]{babel}
\usepackage[utf8]{inputenc}
\usepackage{lmodern}
\usepackage[T1]{fontenc}

\usepackage[usenames]{color}
\usepackage{amssymb, amsthm, amsmath,enumerate,yhmath}
\usepackage{enumitem}
\usepackage{bbm,dsfont}
\usepackage[all]{xy}

\newcommand{\K}{\mathbb{K}}
\newcommand{\R}{\mathbb{R}}
\newcommand{\C}{\mathbb{C}}
\newcommand{\N}{\mathbb{N}}
\newcommand{\A}{\mathcal{A}}
\newcommand{\U}{\mathcal{U}}

\newcommand{\D}{\mathcal{D}}

\newcommand{\V}{\mathcal{V}}
\newcommand{\W}{\mathcal{W}}
\renewcommand{\L}{\mathcal{L}}
\newcommand{\T}{\mathcal{T}}

\renewcommand{\d}{\mathrm{d}}
\newcommand{\I}{\mathrm{I}}
\renewcommand{\Re}{\operatorname{Re}}

\newcommand{\frF}{\mathfrak{F}}
\newcommand{\frA}{\mathfrak{A}}
\newcommand{\frB}{\mathfrak{B}}
\newcommand{\frC}{\mathfrak{C}}

\newcommand{\frD}{\mathfrak{D}}
\renewcommand{\mid}{\, \vert \,}

\renewcommand{\H}{\mathcal{H}}

\DeclareMathOperator{\id}{id}

\DeclareMathOperator{\ran}{Ran}
\newcommand\ra{\rightarrow}
\newcommand\lra{\rightarrow}

\theoremstyle{plain}
\newtheorem{theorem}{Theorem}[section]
\newtheorem{proposition}[theorem]{Proposition}
\newtheorem{lemma}[theorem]{Lemma}
\newtheorem{corollary}[theorem]{Corollary}
\newtheorem{remark}[theorem]{Remark}
\theoremstyle{definition}

\newtheorem{example}[theorem]{Example}
\newtheorem{definition}[theorem]{Definition}

\newtheorem{assumption}[theorem]{Assumption}

\begin{document}
\title[Well-posedness for non-autonomous passive boundary control systems]{Well-posedness of infinite-dimensional non-autonomous passive boundary control systems}
\author{ Birgit Jacob$^1$ and Hafida Laasri$^{2}$
}
\address{}
\email{}
\thanks{ \noindent $1$ University of Wuppertal, School of Mathematics and Natural
Science, Gau\ss strasse 20, D-42119 Wuppertal, Germany, bjacob@uni-wuppertal.de.  \\  $2$  University of Wuppertal, School of Mathematics and Natural
Science, Gau\ss strasse 20, D-42119 Wuppertal, Germany, laasri@uni-wuppertal.de. Support  by  Deutsche Forschungsgemeinschaft (Grant LA 4197/1-1) is gratefully acknowledged.}
\maketitle

\begin{abstract}\label{abstract}
We study a class of non-autonomous boundary control and observation linear systems  that are governed by non-autonomous multiplicative perturbations. This class is motivated by different fundamental partial differential equations, such as controlled wave equations and Timoshenko beams. Our main results give sufficient condition for well-posedness, existence and uniqueness of classical and mild solutions.  
\end{abstract}

 \medskip


\section{Introduction\label{s1}}

We consider the following non-autonomous partial differential equation with boundary input $u$ and  boundary output $y$ 
\begin{align*}
 \frac{\partial}{\partial t}x(t,\zeta)&=\sum_{k=1}^N P_k(t) \frac{\partial^k }{\partial \zeta^k}\big[\H(t,\zeta)x(t,\zeta)\big]+P_0(t,\zeta)\H(t,\zeta)x(t,\zeta),\quad t\geq 0,\ \zeta\in(a,b)
\\ x(0,\zeta)&=x(\zeta), \qquad\qquad \zeta\in(a,b),
\\  u(t)&=W_{B,1} \tau(\H x)(t),\qquad \qquad t\geq 0,
\\ 0&= W_{B,2} \tau(\H x)(t),\qquad \qquad t\geq 0,
\\ y(t)&=W_C \tau(\H x)(t),\quad\quad\quad \quad  t\geq 0.
\end{align*}

 Here $\tau$ denotes the  \textit{trace operator} $\tau: H^N ((a,b);\K^n)\lra \K^{2Nn}$ defined by
 \[ \tau(x):=\big( x(b),x'(b),\cdots, x^{N-1}(b),x(a),x'(a)\cdots, x^{N-1}(a)\big),\]
  $P_k(t)$ is $n\times n$ matrix for all $t\geq 0$, $k=0,1,\cdots,N$, $\H(t,\zeta)\in\K^{n\times n}$ for all $t\geq 0$ and almost every $\zeta\in[a,b]$, $W_{B,1}$ is a $m\times 2nN$-matrix, $W_{B,2}$ is $(nN-m)\times 2nN$-matrix and $W_C$ is a $d\times 2nN$-matrix. Finally, $u(t)\in \K^{m}$ denotes the input and $y(t)\in\K^{d}$ is the output at time $t$.
\medskip

This partial differential  equation is also known as port-Hamiltonian systems, and  covers the wave equation, the transport equation,  beam equations, coupled beam and wave equations as well as certain networks. 
Autonomous port-Hamiltonian systems, that is when $\H, P_k$ are time-independent, have
been intensively  investigated, see e.g.,  \cite{Jac-Kai18, Jac-Mor-Zwa15, Aug-Jac14, Au16, Jac-Zwa12, LZM05, Vi07, Z-G-M-V10}. The existence of mild/classical solutions with non-increasing energy and well-posedness for autonomous port-Hamiltonian 
systems can in most cases be tested via a simple matrix condition  \cite[Theorem 4.1]{LZM05}.  Well-posedness of linear systems in general is not easy to prove and a necessary condition is that the state operator generates a strongly continuous semigroup. For the class of autonomous port-Hamiltonian systems of first order i.e., $N=1$, this condition is even sufficient under some weak assumptions on $P_1\H,$  see  \cite{LZM05} or \cite[Theorem 13.2.2]{Jac-Zwa12}.  
\medskip

In this paper, we aim to generalize these solvability and well-posedness results to the non-autonomous situation. To our knowledge, in contrast to infinite-dimensional autonomous port-Hamiltonian systems, the non-autonomous counterpart has not been discussed so far.  Motivated by this class we start a systematic study of non-autonomous linear boundary control and observation systems, and in particular those  of the following form
\begin{align}
\label{AbstractNBCS Introdution State eq} \dot x(t)&=\frA(t)M(t)x(t), \quad x(0)=x_0, \ t\geq 0\\
\label{AbstractNBCS Introdution}\frB M(t)x(t)&=u(t),\\
\label{AbstractNBCS Introdution Output Eq} \frC M(t)(t) x(t)&=y(t),
\end{align}
which we denote by $\Sigma_{N,M}(\frA, \frB,\frC).$ Here $\frA(t): \frD\subset X\lra X$ is a linear operator, $u(t)\in U$, $y(t)\in Y$, the \textit{boundary operators} $\frB: D(\frB)\subset X\lra U$ and $\frC: \frD\subset X\lra Y$ are linear such that $\frD\subset D(\frB)$, $X$, $U$ and $Y$ are complex Hilbert spaces and $M(t)\in\L(X)$ for all $t\geq 0.$
Setting 
\begin{equation*}\frA(t) x=\sum_{k=0}^N P_k(t) \frac{\partial^k }{\partial \zeta^k}x, \quad 
 \frB x:= \begin{bmatrix}W_{B,1}\\W_{B,2}\end{bmatrix}\tau(x), \quad \frC x:= W_C  \tau(x), \text { and }\ \H(t,\cdot):=M(t)
\end{equation*}
we see  that the non-autonomous port-Hamiltonian system is in fact a special class of non-autonomous systems of the form \eqref{AbstractNBCS Introdution State eq}-\eqref{AbstractNBCS Introdution Output Eq}.

A pair $(x,y)$ is  a \textit{classical solution of} \eqref{AbstractNBCS Introdution State eq}-\eqref{AbstractNBCS Introdution Output Eq} if $x\in C^{1}((0,\infty);X)\cap C([0,\infty);X), y\in C([0,\infty);Y)$ and $x(t)\in D(\frA(t)M(t))$ for all $t\geq 0$ such that $x,y$ satisfy \eqref{AbstractNBCS Introdution State eq}-\eqref{AbstractNBCS Introdution} and \eqref{AbstractNBCS Introdution Output Eq}, respectively. The system $\Sigma_{N,M}(\frA, \frB,\frC)$ is called  \textit{well-posed} if for each (classical) solution $(x,y)$ and any \textit{final time} $\tau>0,$ the operator mapping the input functions $u$ and to the initial state
$x_0$ to $x(\tau)$ and the output functions
$y$ is bounded, i.e.
\begin{equation*}\label{Def: well posedness}
\|x(\tau)\|^2+\int_0^\tau \|y(s)\|^2ds\leq m_\tau \big( \|x_0\|^2+\int_0^\tau \|u(s)\|^2 ds\big)
\end{equation*}
for some constant $m_\tau>0$ independent of $x_0$ and $u.$ 
\medskip

Our approach for the solvability of $\Sigma_{N,M}(\frA, \frB,\frC)$ is based on  a non-autonomous  version of the \textit{Fattorini's trick}, the theory of \textit{evolution families} together with an idea of Schnaubelt and Weiss \cite[Section 2]{SchWei10}. 

Evolution families are a generalization of
strongly continuous semigroups, and are often used to describe the solution of an \textit{abstract non-autonomous Cauchy problem}. In Section \ref{section Intr-EVF}, we 
therefore review the concept of evolution families and that of $C^1$-well posed non-autonomous Cauchy problems. Furthermore, we provide several abstract results which are crucial for the analysis of our non-autonomous boundary control and observation systems.

Fattorini's trick  is well known for autonomous boundary control systems \cite{Jac-Zwa12, Fat68, Cur-Zwa95}. The basic idea of this approach is to reformulate the state and the control equation into an abstract inhomogeneous Cauchy problem on $X.$  A brief description of the autonomous situation is given in Subsection \ref{subsection1 BCOAutobonous}. In Subsection \ref{subsection ExistenceClSolN} we provide a generalization to non-autonomous boundary control systems (see Proposition \ref{Proposition Fattorini Trick Non-A}). This generalization and the results of Section \ref{section Intr-EVF} are then used to prove our main classical solvability results: Theorem \ref{Main theorem pert} and Theorem \ref{main result Phs1}.

The second main purpose of this paper is the study of  the well-posedness for  non-autonomous boundary and observation systems $\Sigma_{N,M}(\frA, \frB,\frC).$ However,  we will restrict ourselves to the case where for every $t\ge 0$ the (unperturbed) autonomous system $\Sigma_{N,\id}(\frA(t), \frB,\frC)$ is \textit{$(R(t),P(t),J(t))$-scattering  passive} i.e., when 
\begin{equation*}2\Re (\frA(t)x\mid P(t)x)_{_X}\leq (R(t)u\mid \frB x)_{_U}-(\frC x\mid J(t)\frC(x)_{_Y} \end{equation*}
for all $x$ in an appropriate subspace of $X\times U$ where $P(t), R(t)$ and $J(t)$ are bounded linear operators. A precise definition and a characterization of scattering passive autonomous and non-autonomous systems is the subject of Subsection \ref{Section: Scattering passive ABCOs} and Subsection \ref{Subsection Scattering passive NBCO-systems}, respectively. Under additional conditions we then prove in Theorem \ref{Main theorem pert} that the perturbed system $\Sigma_{N,M}(\frA, \frB,\frC)$ is well-posed. In particular, we deduce in Theorem \ref{main result Phs1} that well-posedness for a large class of non-autonomous port-Hamiltonian systems can be checked via a simples matrix condition. 

In the literature most attention has been devoted to autonomous control systems. However,  in view of applications, the interest in non-autonomous systems has been rapidly growing in recent years, see  e.g., \cite{HaaHoaOuh18, Pau17, CheWei15, SchWei10, Ha-Rh-Sch08, BouIdr08, Ha06, Sch02} and the references therein. In particular, a class of scattering passive linear non-autonomous linear systems of the form 
\begin{eqnarray}
 \dot{x}(t) &=& A_{-1}M(t)x(t) + Bu(t), \quad t\ge 0, \quad x(0)=x_0, \label{eq1:abst-introduction}\\
 y(t) &=& CM(t)x(t)+ Du(t) \label{eq2:abst-introduction}
 \end{eqnarray}
has been studied by R. Schnaubelt and G. Weiss in \cite{SchWei10}. Here $(A,D(A))$ generates a strongly continuous semigroup on $X$, $A_{-1}\in \L(X,X_{-1})$ is a bounded extension of $(A,D(A))$,  $B\in\L(U,X_{-1}),$  $C\in\L(Z,Y)$ and $D\in \L(U,Y),$ where $X_{-1}$ is the extrapolation space corresponding to $A,$ and $Z:=D(A)+(\alpha-A)^{-1}BU$ for some $\alpha\in\rho(A).$ 

The control part \eqref{AbstractNBCS Introdution State eq}-\eqref{AbstractNBCS Introdution} of the nonautonomous boundary control  system $\Sigma_{N,M}(\frA, \frB,\frC)$ can be  rewritten in the (standard) abstract formulation \eqref{eq1:abst-introduction}, however, in the particular case where $\frA(t)=\frA$ is constant which for non-autonomous port-Hamiltonian system  correspond to the case where the matrices $P_k, k=1,\cdots, N,$ are constant with respect to time variable. On the other hand, when $\frA(t)=\frA$ the output part \eqref{AbstractNBCS Introdution Output Eq} could be also written into \eqref{eq2:abst-introduction} using the concept of \textit{system nodes}. Indeed, well-posed autonomous port-Hamiltonian system fit into the framework of \textit{compatible system nodes} \cite[Theorem 10]{Vi-Go-Zw-Sc05}.  This can be also easily generalized  for boundary control and observation systems defined in Definition \ref{Def:BCOS}. Since we do not follow the approach of \cite{SchWei10}, this topic will not be discussed in this paper and we refer to \cite{Sta-Wei02, TucWei14} for more details on system nodes. 
\medskip

For the general case, that is when $\frA$ is not constant,  then $A_{-1}, B, C, D$ and $Z$ will be time dependent. Thus, the abstract results in  \cite{SchWei10} cannot be immediately applied to deduce classical solvability and well-posedness for \eqref{AbstractNBCS Introdution State eq}-\eqref{AbstractNBCS Introdution Output Eq}.  We expect that the results in  \cite{SchWei10} can be generalized to include this general case. However, for the class of boundary control systems defined in Definition \ref{Def:BCOS} we deal directly with \eqref{AbstractNBCS Introdution State eq}-\eqref{AbstractNBCS Introdution} in combination with Fattorini's trick instead of its corresponding system \eqref{eq1:abst-introduction}-\eqref{eq2:abst-introduction}. Our method is indeed much simpler. Moreover, in general it is not clear how the solution of  \eqref{eq1:abst-introduction}-\eqref{eq2:abst-introduction} can be related to that of  \eqref{AbstractNBCS Introdution State eq}-\eqref{AbstractNBCS Introdution Output Eq} even for the special case where $\frA(t)=\frA$ is constant. In the autonomous case  this relationship is quite simple as we can see in Section \ref{BCOAutobonous}. The reason is that $C_0$-semigroups can be always extended to the extrapolation space. The situation  is more delicate for the non-autonomous setting. Indeed, a general extrapolation theory for evolution families is still missing. Moreover, the extrapolation space may also depend on the time variable. In Section \ref{Section Mild solution} we deal with this question by associating  a \textit{mild solution} to the control part  \eqref{AbstractNBCS Introdution State eq}-\eqref{AbstractNBCS Introdution} of the nonautonomous boundary control  system $\Sigma_{N,M}(\frA, \frB,\frC).$

Finally, we apply our abstract results to non-autonomous port-Hamiltonian systems, in particular to the time-dependent vibrating string and the time-dependent Timoschenko beam. 

\section{Background on evolution families and preliminary results}\label{section Intr-EVF}
Throughout this section $(X, \|\cdot\|)$ is a Banach space.  Let $\A:=\{A(t) \mid t\geq 0\}$ be a family of linear, closed operators with domains $\{D(A(t))\mid t\geq 0\}$. Consider the \textit{non-autonomous Cauchy problem} 
\begin{equation}\label{Nonautonomous InitialValueProblem}
\dot u(t)=A(t)u(t)\  \ {\rm  \  on  } \ [s,\infty),\ u(s)=x_s, (s>0).
\end{equation}
A continuous function $u:[s,\infty)\lra X$ is called a \textit{classical solution} of (\ref{Nonautonomous InitialValueProblem}) if $u(t)\in D(A(t))$ for all $t\geq s, u\in C^1((s,\infty), X)$ and $u$ satisfies (\ref{Nonautonomous InitialValueProblem}). 
\begin{definition}\label{Definition well posedness} The non-autonomous Cauchy problem (\ref{Nonautonomous InitialValueProblem}) is called $C^1$-\textit{well posed} if there is a family $\{Y_t\mid  t\geq 0\}$ of dense subspaces of $X$ such that:
\begin{itemize}
	\item [$(a)$] $Y_t\subseteq D(A(t))$ for all $t\geq 0$.
	\item [$(b)$] For each $s\geq 0$ and $x_s\in Y_s$ the Cauchy problem \eqref{Nonautonomous InitialValueProblem} has a unique classical solution $u(\cdot,s,x_s)$ with $u(t,s,x_s)\in Y_t$ for all $t\geq s$.
	\item [$(c)$] The solutions depend continuously on the initial data $s, x_s$.
\end{itemize} 
If we want to specify the \textit{regularity subspaces} $Y_t$, $t\geq 0$,  we also say (\ref{Nonautonomous InitialValueProblem}) is  $C^1$-\textit{well posed} on $Y_t$.
\end{definition}
\noindent In the autonomous case, i.e., if $A(t)=A$ is constant, then it is well known that the associated Cauchy problem is well-posed if and only if $A$ generates a $C_0$-semigroup $(T(t))_{t\geq 0}$. In this case, for each $x\in D(A)$ the unique classical solution  is given by $T(\cdot)x$.  The following definition provides a natural generalization of operator semigroups for non-autonomous evolution equations.
\par\noindent 

\begin{definition}\label{Definition EVF}
	A family $\U:=\{U(t,s)\mid (t,s)\in\Delta\} \subset\L(X)$ where $\Delta:=\{t,s\geq 0\mid t\geq s\}$ is called \textit{an evolution family} if:
	\begin{itemize}
		\item[$(i)$] $U(t,t)=I$ and $U(t,s)= U(t,r)U(r,s)$ for every $0\leq s\leq r\leq t$, 
		\item [$(ii)$]  $U(\cdot,\cdot): \Delta\lra \L(X)$ is strongly continuous.
		\end{itemize}
The evolution family $\U$ is said to be generated by $\A$, if there is a family $\{Y_t\mid  t\geq 0\}$ of dense subspaces of $X$ with $Y_t\subset D(A(t))$ and 	\begin{itemize}
		\item [$(iii)$]  For every $x_s\in Y_s$, the function $t\lra U(t,s)x_s$ is the unique  classical solution of (\ref{Nonautonomous InitialValueProblem}).
	\end{itemize}
	\end {definition}
\par\noindent The Cauchy problem (\ref{Nonautonomous InitialValueProblem}) is then $C^1$-well posed if and only if  $A(t)$, $t\geq 0$, generates a unique evolution family, see  \cite[Proposition 9.3]{EnNa00} or \cite[Proposition 3.10]{GN96}. Clearly, if $(T(t))_{t\geq 0}$ is a $C_0$-semigroup in $X$ with generator $(A,D(A))$, then $U(t,s):=T(t-s)$ yields an evolution family on $X$ with regularity spaces $Y_t=D(A)$.

\subsection{Similar evolution families}  Let  $\U:=\{U(t,s)\mid  (t,s)\in\Delta\}$ be an evolution family on $X$ and let $\{Q(t)\mid t\geq 0 \} \subset \L(X)$ be a family of isomorphisms on $X$ such that  $Q$ and  $Q^{-1}$ are strongly continuous on $[0,\infty)$. Define the two parameters operator family $\mathcal W:=\{W(t,s)\mid (t,s)\in\Delta\}$ by
\begin{equation}\label{SimilarEVFDef}
W(t,s)=Q^{-1}(t)U(t,s)Q(s)\quad \text{ for } (t,s)\in\Delta.
\end{equation}

It is well known that if $S$ is  a $C_0$-semigroup on $X$ with generator $A$ and $Q\in\L(X)$ is an isomorphism, then $T(\cdot):=Q^{-1}S(\cdot)Q$ is again a $C_0$-semigroup on $X$,  called similar $C_0$-semigroup to $S$, and its generator is given by $Q^{-1}AQ$, where
\[D( Q^{-1}AQ)=D(AQ)=\{x\in X\mid Qx\in D(A)\}={ Q^{-1}D(A)}.\] 
The purpose of this section is to generalize the concept of similar semigroups  to evolution families.
\begin{lemma}The  two parameters family $\W$, defined by  \eqref{SimilarEVFDef}, defines an evolution family on $X$.
\end{lemma}
\begin{proof}
Clearly, the evolution law $(i)$ in Definition \ref{Definition EVF} is fulfilled. It remains to prove the strong continuity of $\W$ in $\Delta$. Let $x\in X$ and $T>0$, and set $\Delta_T:=\{(t,s)\in[0,T]^2 \mid t\geq s\}$. Let $(t,s)$, $(t_n,s_n) \in \Delta_T$ for $n\in\N$ such that $(t_n,s_n)\ra (t,s)$. Then $\{Q^{-1}(t_n)\mid  n\in\N\}$ is bounded by the  uniform boundedness theorem.  Since 
\begin{align*}
\|Q^{-1}(t_n)U(t_n,s_n)x-Q^{-1}(t)U(t,s)x\|&\leq \|Q^{-1}(t_n)\|\|U(t_n,s_n)x-U(t,s)x\|\\&\qquad \qquad+\|[Q^{-1}(t_n)-Q^{-1}(t)]U(t,s)x\|,
\end{align*}
we deduce that $(t,s)\mapsto Q^{-1}(t)U(t,s)x$ is continuous on $\Delta_T$. Thus, using a similar argument for $Q(s)$ and $Q^{-1}(t)U(t,s)$ we obtain that $(t,s)\mapsto W(t,s)x$ is continuous on $\Delta_T$. Since $T>0$ is arbitrary, this proves the assertion. \end{proof}
In contrast to semigroups, the evolution law $(i)$ and the strong continuity $(ii)$ do not guarantee that the given evolution family is generated by some family of linear closed operators.

\begin{proposition}\label{Prop Similar EVF} Assume that  $Q(\cdot)$ is in addition strongly $C^1$-differentiable. Then $\U$ is generated by a family  $\A$ with regularity spaces $\{Y_t\mid  t\geq 0\}$ if and only if $\W$ is generated by $\A_Q:=\{Q^{-1}(t)A(t)Q(t)-Q^{-1}(t)\dot Q(t)\mid \ t\geq 0\}$ with regularity spaces $\{\tilde{Y_t}\mid \ t\geq 0\}$ where \[\tilde{Y_t}:=\{ x\in X\mid {Q}(t)x\in Y_t\}.\]
\end{proposition}
\begin{proof} $(i)$ Assume that $\U$ is generated by $\A$ with regularity spaces $\{Y_t\mid \  t\geq 0\}$. 
We first remark that $\tilde{Y_t}$ is a dense subspace of $X$ and 
\begin{equation}\label{LemmaSimilarEVFProofEq0}
\tilde{Y_t}=Q^{-1}(t)Y_t\subset Q^{-1}(t)D(A(t))=D(A(t)Q(t))=D(A_Q(t))
\end{equation}
for every $t\geq 0$, where $A_Q(t):=Q^{-1}(t)A(t)Q(t)-Q^{-1}(t)\dot Q(t)$. Next, let $x_s\in \tilde{Y_s}$. Then $Q(s)x_s\in Y_s$ and by assumption $U(\cdot,s)Q(s)x_s$ is  the unique classical solution of 
\begin{equation}\label{LemmaSimilarEVFProofEq1}
\dot u(t)=A(t)u(t)\  \ {\rm    on  } \ [s,\infty),\ u(s)=Q(s)x_s, (s>0).
\end{equation}
It follows that $W(t,s)x\in Y_t\subset D(A_Q(t))$ by \eqref{LemmaSimilarEVFProofEq0} and 
\begin{align}
\nonumber\frac{d}{dt} W(t,s)x_s&=[\frac{d}{dt}Q(t)^{-1}]U(t,s) Q(s)x_s+Q(t)^{-1}\frac{d}{dt} U(t,s) Q(s)x_s
\\\label{LemmaSimilarEVFProofEq2}&=-Q(t)^{-1}\dot Q(t)Q(t)^{-1}U(t,s)Q(s)x_s+Q(t)^{-1}A(t)U(t,s) Q(s)x_s
\\\label{LemmaSimilarEVFProofEq3}&=\big[Q^{-1}(t)A(t)Q(t)-Q^{-1}(t)\dot Q(t)\big]W(t,s)x_s.
\end{align}
Since $Q$ is strongly $C^1$-differentiable, it now follows from \eqref{LemmaSimilarEVFProofEq2}-\eqref{LemmaSimilarEVFProofEq3} that $W(\cdot,s)x_s\in C^1((s,\infty), X)$ and $W(\cdot,s)x_s $ solves the non-autonomous problem 
\begin{equation}\label{LemmaSimilarEVFProofEq4}
\dot u(t)=A_Q(t)u(t)\  \ {\rm    on  } \ [s,\infty),\ u(s)=x_s.
\end{equation}
Clearly, $W(\cdot,s)x_s$ is the unique classical solution of \eqref{LemmaSimilarEVFProofEq4}. We conclude that $\W$ is generated by $\{A_Q(t)\mid\ t\geq 0\}$ with regularity space $\{\tilde{Y_t}\mid t\geq 0\}$. 
\par
$(ii)$ Conversely, assume that $\A_Q$ generates the evolution family $\W$ with some regularity spaces $\{\tilde Y_t\mid \  t\geq 0\}.$ Since $Q^{-1}$ is $C^1$-strongly continuous we obtain by $(i)$ that the family $(\A_Q)_{Q^{-1}}=\A$ generates the evolution $\mathcal V$ defined by  \[V(t,s):=Q(t)W(t,s) Q^{-1}(s)=U(t,s), \quad (t,s)\in \Delta\]
with regularity space $Y_t=Q(t)\tilde Y_t.$ This completes the proof.
\end{proof}
 If $A: D(A)\subset X \lra X$ is the generator of a $C_0$-semigroup and $B\in\L(X)$, then the perturbed operator $\tilde A:=A+B$ is again the generator of a $C_0$-semigroup, see e.g., \cite[Section 1.2]{EnNa00} or \cite{Pa83}. This perturbation results fails to be true in general for non-autonomous evolution equations \cite[Example 9.2]{EnNa00}. Thus one cannot conclude from Proposition \ref{Prop Similar EVF} that the family $\{Q^{-1}(t)A(t)Q(t) \mid t\geq 0\}$  generates an evolution family. Nevertheless, inspired by an idea of Schnaubelt and Weiss \cite{SchWei10},  using Proposition \ref{Prop Similar EVF} we show that a positive answer can be given under some additional regularity assumptions.


%
For this we first need  to introduce the following definition.
\begin{definition}(Kato's class) \begin{enumerate}
\item A family $\A$ is said to be \textit{Kato-stable} if  for each $t\geq 0$ there exists a norm $\|\cdot \|_t$ on $X$ equivalent to the original one such that  for each $T\geq 0$ there exists a constant $c_T\geq 0$ with
\begin{equation}\label{stability test condition}
|\|x\|_t-\|x\|_s|\leq c_T|t-s|\|x\|_s, \quad x\in X, t,s\in [0,T ]
\end{equation}
and $A(t)$ generates a contractive $C_0$-semigroup  on $X_t:=(X, \|\cdot\|_t)$ for all $t\geq 0$. 
\item A family $\A$ is said to belong to \textit{Kato's class} if it is Kato-stable and the operators $A(t)$, $t\geq 0$, have a common domain $D$ such that $A(\cdot):[0,\infty)\lra \L(D,X)$ is strongly $C^1$-differentiable.
\end{enumerate}
\end{definition}
 It is known that  Kato-stability  is a sufficient condition for  $C^1$-well posedness of (hyperbolic) non-autonomous evolution equations \cite{Ki64, Pa83,Ta79}. In particular, each non-autonomous evolution equation that is governed by a Kato-class family is $C^1$-well posed.
\medskip

\par\noindent Obviously, $\A$ is Kato-stable if each operator $A(t)$ generates a contractive  $C_0$-semigroup, as one can simply choose  $\|\cdot\|_t=\|\cdot\|$, $t\geq 0$. In this case we  say that $\A$ belongs to the \textit{elementary Kato class}. Starting from this simple case many less trivial  Kato-stable families can be constructed.
\begin{example}\label{main Example Kato's class} Assume that $(H,\|\cdot\|_H)$ is a Hilbert space. 
Let $M:[0,\infty)\lra \L(H)$ be \textit{self-adjoint and uniformly coercive}, i.e., $M(t)^*=M(t)$ and $(M(t)x|x)_H\geq \beta\|x\|^2_H$ for some constant $\beta>0$ and all $t\geq 0$. If $M$ is strongly $C^1$-continuous and $M^{-1}$ is strongly continuous, then for each $t\in[0,\infty)$ the function 
\begin{equation}\label{Eq Equivalent norm}x\mapsto\|x\|_{t}:=\sqrt{(M(t)x|x)}=\|M^{1/2}(t)x\|
\end{equation}
defines a norm on $H$ which is equivalent  to the norm $\|\cdot\|_H$ and satisfies \eqref{stability test condition}. Moreover, if $\A$ has a common domain $D$ and for each $t\geq 0$ the operator $(A(t),D)$ generates a contraction $C_0$-semigroup in $H$, then  $(A(t)M(t), D(A(t)M(t))$ and $(M(t)A, D(A(t)))$ generate  contractive $C_0$-semigroups on $H_t$, and thus both families $\{A(t)M(t)\mid t\geq 0\}$ and  $\{M(t)A(t)\mid t\geq 0\}$ are Kato-stable. We refer to \cite[Lemma 7.2.3]{Jac-Zwa12} and to the proof of \cite[Proposition 2.3]{SchWei10} for precise details. Finally, if $P:[0,\infty)\lra \L(X)$ is a locally uniformly bounded function, then $\{M(t)A(t)+P(t)\mid t\geq 0\}$ and $\{A(t)M(t)+P(t)\mid t\geq 0\}$ are Kato-stable \cite[Propositions 4.3.2 and 4.3.3]{Ta79}.
\end{example}
\begin{proposition}\label{main proposition 2} Let $\A$ belong to the Kato-class and let  $D$ denote the common domain of $A(t)$, $t\geq 0$. Assume that $Q(\cdot)$ is strongly $C^2$-continuous. Then $\{Q^{-1}(t)A(t)Q(t)\mid \ t\geq 0\}$ generates a unique evolution family $\W$ with regularity spaces $Y_t= Q{{-1}}(t)D$, $t\geq 0$. Moreover,  for each $F\in C^1([0,\infty); X)$ and $x_s\in Q{^{-1}}(s)D$ the inhomogeneous non-autonomous Cauchy problem 
\begin{equation}\label{InhNonautonomous InitialAbstValueProblemIntermediat}
\dot x(t)=Q^{-1}(t)A(t)Q(t)x(t)+F(t)\  \ {\rm   a.e.\  on  } \ [s,\infty),\ x(s)=x_s, s>0,
\end{equation}
has a unique classical solution given by 
\begin{equation}\label{VariContFormula}
x(t)=W(t,s)x_s+\int_s^t W(t,r)F(r) \d r \quad t\geq s.
\end{equation}
\end{proposition}

\begin{proof} It is not difficult to verify that \eqref{stability test condition} implies that   $\|x\|_t\leq e^{c_T|t-s|}\|x\|_s$ for all $x\in X$, $t,s\in [0,T]$ and $T>0$. Using \cite[Propositions 4.3.2 and 4.3.3]{Ta79} and \cite[Corollary of Theorem 4.4.2]{Ta79} we obtain that $\{A(t)+\dot Q(t)Q^{-1}(t)\mid t\geq 0\}$ generates a unique evolution family $\U$ on $X$. Thus the first assertion follows from Proposition \ref{Prop Similar EVF}.  Next, let $F\in C^1([0,\infty); X)$ and $x_s\in Q(s)D$. By \cite[Theorem 4.5.3]{Ta79} the inhomogeneous Cauchy problem 
\begin{align}\label{intermediat problem Inh}
\dot u(t)&=A(t)u(t)+\dot Q(t)Q^{-1}(t)u(t)+Q(t)F(t)\  \ {\rm   a.e.\  on  } \ [s,\infty),\\ u(s)&=Q^{-1}(s)x_s, s>0.
\end{align}
has a unique classical solution $x$ given by 
\begin{equation}\label{VariContFormula2}
u(t)=U(t,s)Q^{-1}(s)x_s+\int_s^t U(t,r)Q(r)F(r) \d r \quad t\geq s.
\end{equation}

\par\noindent On the other hand,  arguing as in the proof of Proposition \ref{Prop Similar EVF} we see that  $x:=Q^{-1}(\cdot)u$ is a classical solution of \eqref{InhNonautonomous InitialAbstValueProblemIntermediat}. The uniqueness of classical solutions of \eqref{InhNonautonomous InitialAbstValueProblemIntermediat} follows from the uniqueness of classical solutions of \eqref{intermediat problem Inh}. Finally, \eqref{VariContFormula} follows from  \eqref{VariContFormula2} and \eqref{SimilarEVFDef}.
\end{proof}
Using Example \ref{main Example Kato's class} and Proposition \ref{main proposition 2} one can formulate the following two corollaries.
\begin{corollary}\label{main corollary Abstract} Assume that $X$ is a Hilbert space. Assume that $\A$ belongs to the elementary Kato class and denote by $D$ the common of $A(t)$, $t\geq 0$. Let  $M:[0,\infty)\lra \L(X)$  and $P:[0,\infty)\lra \L(X)$ be self-adjoint and uniformly coercive such that $M$  is strongly $C^2$-continuous while $P$ is strongly $C^1$-differentiable. Then $\{A(t)M(t)+P(t)\mid \ t\geq 0\}$ generates a unique evolution family $\W$ with regularity spaces $Y_t= {M^{-1}}(t)D$, $t\geq 0$. Moreover,  for each $F\in C^1([0,\infty); X)$ and $x_s\in { M^{-1}}(s)D$ the inhomogeneous non-autonomous Cauchy problem 
\begin{equation}\label{InhNonautonomous InitialAbstValueProblemIntermediat2}
\dot x(t)=A(t)M(t)x(t)+P(t)x(t)+F(t)\  \ {\rm   a.e.\  on  } \ [s,\infty),\ x(s)=x_s, s>0.
\end{equation}
has a unique classical solution given by \eqref{VariContFormula}.
\end{corollary}
\begin{proof} For the proof we just have to apply Proposition \ref{main proposition 2} for $M(t)A(t)+M(t)P(t)M^{-1}(t)$ instead of $A(t)$ and $M(t)$ instead of $Q(t)$.
\end{proof}
\begin{corollary}\label{main corollary Abstract2} Let $X$ be a Hilbert space and let $(A,D(A))$ be generator of a contractive $C_0$-semigroup on $X$.  Let  $M:[0,\infty)\lra \L(X)$  and $P:[0,\infty)\lra \L(X)$ be as in Corollary \ref{main corollary Abstract}. Further, let $R:[0,\infty)\lra \L(X)$ be self-adjoint and uniformly coercive such that $R$ is strongly $C^1$-continuous and commute with $M$  i.e. 
\begin{equation}
R(t)M(t)=M(t)R(t)\quad\text{ for all} \ t\geq 0.
\end{equation} 
Then the family $\{R(t)AM(t)+P(t)\mid \ t\geq 0\}$ generates a unique evolution family $\W$ with regularity spaces $Y_t= {M^{-1}}(t)D(A)$, $t\geq 0.$ Moreover,  for each $F\in C^1([0,\infty); X)$ and $x_s\in { M^{-1}}(s)D(A)$ the inhomogeneous non-autonomous Cauchy problem 
\begin{equation}\label{InhNonautonomous InitialAbstValueProblemIntermediat2}
\dot x(t)=R(t)AM(t)x(t)+P(t)x(t)+F(t)\  \ {\rm   a.e.\  on  } \ [s,\infty),\ x(s)=x_s, s>0.
\end{equation}
has a unique classical solution given by \eqref{VariContFormula}.
\end{corollary}
\begin{proof}From Example \ref{main Example Kato's class} we deduce that the family $\{M(t)R(t)A\mid t\geq 0 \},$ and therefore $\{M(t)R(t)A+M(t)P(t)M^{-1}(t)\mid t\geq 0 \},$ belongs to Kato's class. In fact, using  \eqref{InhNonautonomous InitialAbstValueProblemIntermediat2} we see that $M(\cdot)R(\cdot):[0,\infty)\lra \L(X)$ is selfadjoint and uniformly coercive.  Now, applying Proposition \ref{main proposition 2} for $M(t)R(t)A+M(t)P(t)M^{-1}(t)$ instead of $A(t)$ and $M(t)$ instead of $Q(t)$ concludes the proof.
\end{proof}
\begin{remark} Corollary  \ref{main corollary Abstract} has been proved in \cite[Proposition 2.8-(a)]{SchWei10} using a slightly different method for $A(t)=A$ and $F=P=0$.\end{remark}
\subsection{Backward evolution families}
Let $X$ be a Hilbert space over $\K=\C$ or $\R$.
\begin{definition}\label{Backward EVF} A family $\V:=\{V(t,s)\mid \ (t,s)\in\Delta\}\subset \L(X)$ is called a \textit{backward evolution family}  if \begin{itemize}
		\item[$(i)$] $V(t,t)=I$ and $V(r,s)V(t,r)=V(t,s) $ for every $0\leq s\leq r\leq t$, 
		\item [$(ii)$]  $V(\cdot,\cdot): \Delta\lra \L(X)$ is strongly continuous.
		\end{itemize}
A family $A(t): D(A(t))\subset X\lra X$, $t\geq 0$, of linear operators generates a backward evolution equation $\V$ if  there is a family $\{Y_{t}\mid  t\geq 0\}$ of dense subspaces of $X$ with $Y_{t}\subset D(A(t))$ and 
\begin{equation}
V(t,s)Y_t \subset Y_s \text{ for all } \ 0\leq s\leq t,
\end{equation}
$V(t,\cdot)x_t\in C^1([0,t], X)$ for every $x_t\in Y_t$ and $V(t,\cdot)x_t$ solves uniquely the \textit{backward} non-autonomous problem 
\begin{align}
\label{BackwardCP}\dot u(s)=-A(s)u(s)\  \ \text{  on  } \ 0\leq s\leq t ,\ u(t)=x_t, (t>0).
\end{align}
\end{definition}

\begin{lemma}\label{Lemma Back-ForwardEVFRelation}
 \begin{enumerate} 
\item Assume that $\A=\{A(t)\mid t\geq 0\}$ belongs to the elementary Kato-class. Then $\A$  generates a backward evolution family.
\item Assume that $\A$ generates an evolution family $\U.$  If the adjoint operators $\A^*:=\{A^*(t) \mid t\geq 0\}$ generate a backward evolution family $\U_*:=\{U_*(t,s)\mid \ (t,s)\in\Delta\},$ then for $(t,s)\in\Delta$ we have 
\begin{equation}\label{Back-ForwardEVFRelation} U(t,s)=[U_*(t,s)]^\prime.\end{equation} 
\end{enumerate}
\end{lemma}
\begin{proof}
$(i)$ Let $T>0$ be fixed and set $\A_T:=\{A(T-t)\mid t\in[0,T]\}$. Then, obviously $\A_T$   belongs to the Kato-class and thus generates an evolution family $\U_T:=\{U_T(t,s)\mid 0\leq s\leq t\leq T\}$ \cite[Theorem 4.8]{Pa83} such that for each $x\in D$ and $0\leq s\leq t\leq T$
\begin {align}
\frac{\d }{\d t} U_T(t,s)x&=A_T(t)U_T(t,s)x,\\
\frac{\d}{\d s} U_T(t,s)x&=-U_T(t,s)A_T(s)x.
\end{align}
It is easy to see that $S(t,s):=U_T(T-s,T-t)$ for each $0\leq s\leq t\leq T$ defines a backward evolution family with generator $\{A(t)\mid t\in[0,T]\}$. This completes the proof since $T$ is arbitrary.
\par\noindent $(ii)$ Denote by $Y_t$ and $Y_{t,*}, t\geq 0$ the  regularity spaces corresponding to $\A$ and $\A^*$, repectively. Let $t>s\geq 0$ and let $x_s\in Y_{s}$ and $y_t\in Y_{t,*}$. Then  for $s\geq r\geq t$ we have
\begin{align*}\frac{\d}{\d r}( x_s\mid [U(r,s)]^{\prime}U_*(t,r)y_t)&=\frac{\d}{\d r}( U(r,s)x_s\mid U_*(t,r)y_t)
\\&=(A(r)U(r,s)x_s\mid U_*(t,r)y_t)-(U(r,s)x_s\mid A^*(r)U_*(t,r)y_t)\\&=0.
\end{align*}
Integrating over $[s,t]$ and using that $Y_s$ and $Y_{t,*}$ are dense in $X$ yield the desired identity.
\end{proof}

\section{Review on Autonomous boundary control and observation systems}\label{BCOAutobonous}
Many systems governed by  linear partial differential equations with inhomogeneous boundary conditions are  described by an  abstract boundary system of the form 
 \begin{align}
\label{StateEqBCS}\dot x(t)&=\frA x(t), \quad x(0)=x_0,\  t\geq 0,\\
\label{ConEqBCS}\frB x(t)&=u(t),\\
\label{ObsEqBCS}\frC  x(t)&=y(t).
\end{align}
Here $\frA: D(\frA)\subset X\lra X$ is a linear operator, $u(t)\in U$, $y(t)\in Y$, the \textit{boundary operators} $\frB: D(\frB)\subset X\lra U$ and $\frC: D(\frA)\subset X\lra Y$ are linear such that $D(\frA)\subset D(\frB)$, and  $X$, $U$ and $Y$ are complex Hilbert spaces. We shall call $X$ the \textit{state space}, $U$ the \textit{input space} and $Y$ the \textit{output space} of the system.
\par\noindent In this section, we  recall some well-known results on  \textit{well-posedness} of  these system which are needed throughout this paper.
 \begin{definition}\label{Definition: ClassSol and wellposedness Autonomous case} Let $x_0\in X$ and $u:[0,\infty)\ra U$ be given. 
\begin{itemize}
\item [$(i)$] $x$ is called a \textit{classical solution} of \eqref{StateEqBCS}-\eqref{ConEqBCS}, if $x\in C^1([0,\infty),X)$, $x(t)\in D(\frA )$ for all $t\geq 0$ and $x$ satisfies \eqref{StateEqBCS}-\eqref{ConEqBCS}.
	\item [$(ii)$] A pair $(x,y)$ is called a \textit{classical solution} of \eqref{StateEqBCS}-\eqref{ObsEqBCS}, if 
$x$ is a classical solution of \eqref{StateEqBCS}-\eqref{ConEqBCS}, $y\in C([0,\infty);Y)$ and $y$ satisfies \eqref{ObsEqBCS}.
\item [$(iii)$] The system $\Sigma(\frA,\frB,\frC)$ is called  \textit{well-posed}, if for any \textit{final time} $\tau>0$ there exists $m_\tau>0$ such that for all classical solution of  \eqref{StateEqBCS}-\eqref{ObsEqBCS} we have  
\begin{equation}\label{DefA: well posedness}
\|x(\tau)\|^2+\int_0^\tau \|y(s)\|^2ds\leq m_\tau \big( \|x(0)\|^2+\int_0^\tau \|u(s)\|^2 ds\big).
\end{equation}
\end{itemize}
\end{definition}
Remark that, if $\frC\in \L(D(\frA), Y)$, then $(x,y)$ is a classical solution of \eqref{StateEqBCS}-\eqref{ObsEqBCS} if and only if $x$ is a classical solution of \eqref{StateEqBCS}-\eqref{ConEqBCS}.
\subsection{Existence of classical solutions}\label{subsection1 BCOAutobonous}
In order to study existence of classical solutions it is often useful to write the boundary control system \eqref{StateEqBCS}-\eqref{ConEqBCS} as a $C^1$-well posed (inhomogeneous) autonomous Cauchy problem. We introduce the following definition which is based on Curtain and  Zwart \cite[Definition 3.3.2]{Cur-Zwa95}.

\begin{definition}\label{Def:BCOS} The linear (autonomous) system \eqref{StateEqBCS}-\eqref{ObsEqBCS} is called \textit{a boundary control and observation autonomous system}, and we write $\Sigma(\frA,\frB,\frC)$ is a BCO-system, if the following assertions hold:
\begin{itemize}
	\item [$(i)$] The operator $A:D(A)\subset X\lra X$, called \textit{the main operator}, defined by
		\begin{align*}D(A):&=D(\frA)\cap \ker(\frB )
		\\ A x:&=\frA  x \quad \text{for}\  x\in D(A)
		\end{align*}
		generates a strongly continuous semigroup on $X$.
   \item [$(ii)$] There exists a linear operator $\tilde B\in  \L(U,X)$ such that for all $u\in U$ we have \[\tilde Bu\in D(\frA), \frA \tilde B\in\L(U,X)\  \text{ and } \ \frB \tilde Bu=u.\]
   \item [$(iii)$] $\frC: D(\frA)\subset X\lra Y$ is a linear bounded operator,  where $D(\frA)$ is equipped with the graph norm.
\end{itemize}
\end{definition} 

\par\noindent In the following $\Sigma(\frA,\frB,\frC)$ is assumed to be  a BCO-system. The following remark will be very useful for non-autonomous boundary control systems. 

\begin{remark}\label{key remark} Let $\Sigma(\frA,\frB,\frC)$ be a BCO-system. Then for each $(x,u)\in X\times U$ we have
\begin{align*} x\in D(\frA)  \text{ and } \frB x=u &\Longleftrightarrow x-\tilde Bu\in D(A).\end{align*}
This is an easily consequence of Definition \ref{Def:BCOS}.
\end{remark}
We denote by $X_{-1}$ the extrapolation space associated to $A$, i.e., the completion of $X$ with respect to the norm $x\mapsto \|(\beta I-A)^{-1}x\|$ for some arbitrary $\beta\in\rho(A)$. Let $A_{-1}$ be the extension of $A$ to $X_{-1}$.  It is well known that $A_{-1}$ with domain $X$ generates a $C_0$-semigroup $(T_{-1}(t))_{t\geq 0}$ on $X_{-1}$ and  for all $t\geq 0$ the operator $T_{-1}(t)$ is the unique continuous extension of $T(t)$ to $X_{-1}$. 
 We associate with $\Sigma(\frA,\frB,\frC)$ the   linear operator $B\in \L(U,X_{-1})$ called  \textit{control operator} defined by 
\begin{equation}\label{control operator for BOC} 
B:=\frA\tilde B-A_{-1}\tilde B.
\end{equation}
It turns out, that for sufficiently smooth initial data and inputs the two Cauchy problems
\begin{align}
     \label{BCsDistributed2}\dot{w}(t) &=Aw(t) +\frA\tilde Bu(t)- \tilde B\dot u(t), \quad t\ge 0, \quad  w(0)=x_0-\tilde Bu(0),\\
     \label{BCsDistributed22}\dot{x}(t) &=Ax(t) +Bu(t), \quad t\ge 0,\quad   x(0)=x_0,
    \end{align}
and the BCO-system \eqref{StateEqBCS}-\eqref{ObsEqBCS} are equivalent. More precisely, we have 
 \begin{proposition}\label{Proposition: BCS-Standard CS}Let $(x_0,u)\in D(\frA)\times W^{2,2}([0,\infty);U)$ such that $\frB x_0=u(0)$. Then \eqref{StateEqBCS}-\eqref{ConEqBCS} has a unique classical solution $x$ given by
\begin{eqnarray}\label{classicalSolFormula}
x(t)&=&T(t)x_0+\int_{0}^t T_{-1}(t-s)Bu(s)ds \qquad t\geq 0,\\
\label{classicalSolFormula2} &=&T(t)(x_0-\tilde Bu(0))+\int_{0}^t T(t-s)\big[\frA\tilde Bu(s)-\tilde B\dot u(s)\big] ds +\tilde Bu(t)
\\\label{classicalSolOutputFormula} y(t)&=&\frC T(t)x_0+\frC\int_{0}^t T_{-1}(t-s)Bu(s)ds \qquad t\geq 0.
\end{eqnarray}
Therefore, $x$ is the unique classical solution of  \eqref{BCsDistributed22} and  $w:=x-\tilde u$ is the unique classical solution of \eqref{BCsDistributed2} with initial value $w_0=x_0-\tilde Bu(0)$.
 \end{proposition}
    
    \begin{proof} The proof follows from a  combination of \cite[Theorem 11.1.2]{Jac-Zwa12} (see also \cite[Theorem 3.3.3]{Cur-Zwa95}) and \cite[Corollary 10.1.4]{Jac-Zwa12}  taking  Remark \ref{key remark} into account.
   \end{proof}
\subsection{Scattering passive BCO-systems}\label{Section: Scattering passive ABCOs} Let $\Sigma(\frA, \frB,\frC)$ be a BCO-system on $(X,U,Y)$ and let $P=P^*\in \L(X)$, $R=R^*\in \L(U)$ and $J=J^*\in \L(Y)$. The \textit{admissible space} $\V\subset X\times U$ is defined by
\begin{align*}\label{admissible space}\mathcal V:&=\left\{(x,u)\in X\times U\ \big|\  x\in D(\frA) \text{ and }  \frB x=u\right\}.
\end{align*}
\begin{definition}\label{Definition scattering passive autonomous} We say that  $\Sigma(\frA,\frB,\frC)$ is \textit{$(P,R,J)$-scattering passive}  if
\begin{equation}\label{Scatterin system: Definition1}
\frac{\d}{\d t} (Px(t)\mid x(t))_X\leq (Ru(t)\mid u(t))_U-(y(t)\mid Jy(t))_Y,\quad \text{ for all } t\geq 0
\end{equation}
and  all classical solutions $(x,y)$ of \eqref{StateEqBCS}-\eqref{ObsEqBCS}. Further, $\Sigma(\frA,\frB,\frC)$  is called  \textit{(R,P,J)-scattering energy preserving}  if equality holds in  \eqref{Scatterin system: Definition1}.  If $P=I, R=I$ and $J=I$, then we simply say that $\Sigma(\frA,\frB,\frC)$ is \textit{scattering passive} (or \textit{dissipative}).
\end{definition}
Each $(P,R,J)$-scattering passive boundary system $\Sigma(\frA,\frB,\frC)$ is well-posed if $P$ and $J$ are invertible. This can be seen by using Gronwall's Lemma (see the proof of Lemma \ref{sufficient condition Wellposedness NBCOs}). The following lemma characterizes $(P,R,J)$-scattering passive BCO-systems. A comparable results has been proved in \cite[Theorem 3.2, Proposition 5.2]{Ma-Sta-Wei06} for  \textit{systems nodes}.
\begin{lemma}\label{Prop: Characterization of passive BCO system} The BCO-system $\Sigma(\frA, \frB,\frC)$ is $(P,R,J)$-scattering passive if and only if for each $(x_0,u_0)\in \V$ we have 
\begin{equation}\label{Eq Characterization of passive BCO system}
2\Re (\frA x_0\mid P x_0)_X\leq (Ru_0\mid \frB x_0)_U-(\frC x_0\mid J \frC x_0)_Y
\end{equation}
or equivalently, 
\begin{equation}\label{Eq 2Characterization of passive BCO system}
2\Re (\frA x_0\mid P x_0)_X\leq (Ru_0\mid u_0)_U-(\frC x_0\mid J \frC x_0)_Y.
\end{equation}
 Then the  BCO-system  $\Sigma(\frA, \frB,\frC)$ is $(P,R,J)$-energy preserving if and only if equality holds in \eqref{Eq Characterization of passive BCO system}, or equivalently in  \eqref{Eq 2Characterization of passive BCO system}.
\end{lemma}
\begin{proof} Obviously, the inequalities \eqref{Eq Characterization of passive BCO system} and \eqref{Eq 2Characterization of passive BCO system} are equivalent since $\frB x_0=u_0$ for each $(x_0,u_0)\in \V$. Assume that  $\Sigma(\frA,\frB,\frC)$ is $(R,P,J)$-scattering passive. Let $(x_0, u_0)\in \mathcal V$ and $u:[0,\infty)\ra U$ such that $u(0)=u_0$.  Assume that $(x,y)$ is  a classical solution of  \eqref{StateEqBCS}-\eqref{ObsEqBCS} corresponding to $(x_0,u)$. Then $(x(t),u(t))\in \V$  and 
\begin{align*}
\frac{\d}{\d t} (Px(t)\mid x(t))&=2\Re (\dot x(t)\mid Px(t))
=2\Re (\frA x(t)\mid P x(t))
\end{align*}
for all $t\geq 0$. Inserting this into \eqref{Scatterin system: Definition1} yields 
\begin{equation}
2\Re (\frA x(t)\mid P x(t))\leq (Ru(t)\mid \frB x(t))_U-(\frC x(t)\mid J\frC x(t))_Y
\end{equation}
for all $t\geq 0$. The previous inequality implies \eqref{Eq Characterization of passive BCO system}  by taking $t=0$. The converse implication and the last assertion can be  proved similarly. 
\end{proof}

\section{Non-autonomous boundary and observation systems}\label{Sec 3}
In this section, our aim  is to extend the results of Section \ref{BCOAutobonous} to the more general case where $\frA,\frB$, and $\frC$ are time dependent. 
Let $X, U$ and $Y$ be Hilbert spaces over $\K=\C$ or $\R$. For each $t\geq 0$ we consider the linear operators $\frA(t):D(\frA(t))\subset X\ra X$,  $\frB(t): D(\frB(t))\subset X\ra U$ and $\frC(t): D(\frA(t))\subset X\lra Y$  such that $ D(\frA(t))\subset D(\frB(t))$ for each $t\geq 0$. 

We consider the following abstract non-autonomous boundary  system
 \begin{align}
\label{NStateEqBCS}\dot x(t)&=\frA(t)x(t), \quad t\geq s, \quad x(s)=x_s,\ (s\geq 0)\\
\label{NConEqBCS}\frB(t)x(t)&=u(t),\quad t\geq s,\\
\label{NObsEqBCS}\frC(t) x(t)&=y(t),\quad t\geq s,
\end{align}
which we denote by $\Sigma_N(\frA,\frB,\frC)$.
\begin{definition}\label{Def: NBCO}
Let $s\geq 0, x_s\in X$ and $u:[0,\infty)\lra U$ be given. 
\begin{itemize}
\item [$(i)$] A function $x:[s,\infty)\lra X$ is called a \textit{classical solution} of \eqref{NStateEqBCS}-\eqref{NConEqBCS}, if $x\in C^1([s,\infty),X)$, $x(t)\in D(\frA(t))$ for all $t\geq s$ and $x$ satisfies \eqref{NStateEqBCS}-\eqref{NConEqBCS}.
\item [$(ii)$] A pair $(x,y)$ is a \textit{classical solution} of \eqref{NStateEqBCS}-\eqref{NObsEqBCS}, if $x$ is classical solution of \eqref{NStateEqBCS}-\eqref{NConEqBCS}, $y\in C([s,\infty);Y)$ and $(x,y)$ satisfies \eqref{NStateEqBCS}-\eqref{NObsEqBCS}.
\item [$(iii)$] $\Sigma_N(\frA,\frB,\frC)$ is a \textit{non-autonomous boundary control and observation system}, and we write NBCO-systems, if for each $t\geq 0$ the autonomous system $\Sigma(\frA(t),\frB(t),\frC(t))$ is a BCO-system such that the family $\{A(t)\mid t\geq 0\}$ of main operators generates an evolution family.
\item [$(iv)$] The non-autonomous system $\Sigma_N(\frA,\frB,\frC)$ is called  \textit{well-posed} if for any \textit{final time} $\tau>0$ there exists a constant $m_\tau>0$ such that for all classical solution of  \eqref{NStateEqBCS}-\eqref{NObsEqBCS} we have  
\begin{equation*}
\|x(\tau)\|^2_X+\int_s^\tau \|y(r)\|_Y\d r\leq m_\tau \Big( \|x(s)\|^2_X+\int_s^\tau \|u(r)\|_U^2 \d r\Big).
\end{equation*}
 
\end{itemize}
\end{definition}

\subsection{Existence of classical solutions}\label{subsection ExistenceClSolN}
Let $\Sigma_N(\frA,\frB,\frC)$ be a NBCO-system. In this subsection, we study existence and uniqueness of classical solutions of $\Sigma_N(\frA,\frB,\frC)$ without output, i.e., classical solution of \eqref{NStateEqBCS}-\eqref{NConEqBCS}.  In the previous section we have seen in the autonomous  case that  \eqref{NStateEqBCS}-\eqref{NConEqBCS}  can be equivalently written as a $C^1$-well-posed inhomogeneous Cauchy problem (in $X$) for sufficiently smooth initial data and inputs. This idea can be  extended to the non-autonomous setting. 
\par For each $t\geq 0$, we denote by $A(t): D(A(t)\subset X\ra X$ the main operator of $\Sigma(\frA(t),\frB(t),\frC(t))$, and by $\U$ the evolution family generated by  $\{A(t)\mid t\geq 0\}$.  Further,  according to Definition \ref{Def: NBCO}-$(iii)$ there exists $\{\tilde B(t)\mid t\geq 0\}\subset \L(U,X)$ such that for all $t\geq 0$ we have \begin{equation}\tilde B(t)U\subset  D(\frA(t)),\  \frA(t) \tilde B(t)\in\L(U,X)\  \text{ and } \ \frB(t) \tilde B(t)=I_{U}. \end{equation}
We also consider the time-dependent admissible spaces $\V(t), t\geq 0$, i.e,
\begin{equation*}
\V(t):=\{(x,u)\in X\times U\mid x\in D(\frA(t)) \text{ and } \frB(t)x=u\}.
\end{equation*}
 Since $\{A(t)\mid t\geq 0\}$ generates an evolution family $\U$ on $X$, for a given $f\in L_{Loc}^1([0,\infty);X)$ the  inhomogeneous non-autonomous Cauchy problem 
 \begin{eqnarray} 
\label{InNCB} \dot{v}(t)&=&A(t)v(t) + f(t), \quad t\ge s,\ (s\geq 0), \  \\\label{InNCB2} v(s)&=&v_s, 
\end{eqnarray}
 has at most one classical solution given by 
\begin{equation*}
v(t)=U(t,s)v_s+\int_s^t U(t,r)f(r)\d r,
\end{equation*}
see e.g., \cite[Section 5.5.1]{Pa83}.
Thus the following  proposition provides a  generalization of \cite[Theorem 3.3.3]{Cur-Zwa95} (see also Proposition \ref{Prop: Characterization of passive BCO system}).
\begin{proposition}\label{Proposition Fattorini Trick Non-A} Assume that $u\in C^1([0,\infty);U)$, $\tilde B(\cdot)u_0\in C^1([0,\infty);X)$  and  $\frA(\cdot)\tilde B(\cdot)u_0\in L^1([0,\infty);X)$ for each $u_0\in U$.  Let $x_s\in X$ such that $(x_s,u_s)\in \V(s)$. Then $x$ is a classical solution of \eqref{NStateEqBCS}-\eqref{NConEqBCS} if and only if $v:=x-\tilde Bu$ is a classical solution of \eqref{InNCB}-\eqref{InNCB2}   with inhomogeneity \begin{equation}\label{ inhomogeneity Fattorini trick}f(t)=\frF_u(t):=\frA(t)\tilde B(t)u(t)-\frac{\d}{\d t} \big[\tilde B(t)u(t)\big]\end{equation}
and initial data $v_s=x_s-\tilde B(s)u(s)$. Therefore, \eqref{NStateEqBCS}-\eqref{NConEqBCS} has at most one classical solution $x$  given by 
\begin{align} \label{classSolNBOC} x(t)&=U(t,s)[x_s-\tilde B(s)u(s)]+\tilde B(t)u(t)+\int_s^tU(t,r)\frF_u(r)\d r
\end{align}
for each $t\geq s$. 
\end{proposition}
\begin{proof} Let $s\geq 0.$
Clearly $x\in C^1([s,\infty); X)$ if and only if $v\in  C^1([s,\infty); X)$. Assume now that $x$ is a classical solution of \eqref{NStateEqBCS}-\eqref{NConEqBCS}. Then $v(t)\in \V_t\subset D(A(t))$ for every $t\geq s$ by Remark \ref{key remark} and 
\begin{align*}
\dot v(t)&=\dot x(t)-\dot{\tilde{B}}(t)u(t)-\tilde B(t)\dot u(t)
\\&=\frA(t)x(t)-\frA(t)\tilde B(t)u(t)+\frA(t)\tilde B(t)u(t)-\dot{\tilde{B}}(t)u(t)-\tilde B(t)\dot u(t)
\\&=A(t)[x(t)-\tilde B(t)u(t)]+\frA(t)\tilde B(t)u(t)-\dot{\tilde{B}}(t)u(t)-\tilde B(t)\dot u(t)\\&=A(t)v(t)+\frA(t)\tilde B(t)u(t)-\frac{\d}{\d t} \big[\tilde B(t)u(t)\big].
\end{align*}
Thus $v$ is a classical solution of \eqref{InNCB} with $f$ given by \eqref{ inhomogeneity Fattorini trick}. The converse implication can be proved similarly. Finally, \eqref{classSolNBOC} follows by the above the remark.
\end{proof}

\subsection{Scattering passive NBCO-systems}\label{Subsection Scattering passive NBCO-systems}
 Let $R:[0,\infty)\lra \L(U)$, $P: [0,\infty)\lra \L(X)$ and $J:[0,\infty)\lra \L(Y)$ be continuous functions such that $P$ is strongly differentiable and $R(t)^*=R(t)$, $P(t)^*=P(t)$, $J(t)^*=J(t)$ for all $t\geq 0$. 
\begin{definition} Let $(x,y)$ be classical solution of \eqref{NStateEqBCS}-\eqref{NObsEqBCS}. Then $\Sigma_N(\frA,\frB,\frC)$ is called \textit{(R,P,J)-scattering passive} if  for all $t\geq s$
\begin{equation}\label{(R,P,J)-scattering passive Definition 1} 
\frac{\d}{\d t} (P(t)x(t)\mid x(t))+(y(t)\mid J(t)y(t))_Y\leq (u(t)\mid R(t)u(t))_U+ (\dot P(t)x(t)\mid x(t)).
\end{equation}
Further, $\Sigma_N(\frA,\frB,\frC)$  is called  \textit{(R,P,J)-scattering energy preserving}  if equality holds in \eqref{(R,P,J)-scattering passive Definition 1}. If $P=I, R=I$ and $J=I$ then $\Sigma_N(\frA,\frB,\frC)$ is called \textit{scattering passive}, and \textit{scattering energy preserving} if we have equality in \eqref{(R,P,J)-scattering passive Definition 1}. 
\end{definition}
We have seen in Section \ref{Section: Scattering passive ABCOs} that for autonomous BCO-systems $\Sigma(\frA,\frB,\frC)$ the $(R,P,J)$-scattering passivity can be characterized in terms of $\frA$, $\frB$ and $\frC$ and $(R,P,J)$-scattering passivity is a sufficient condition for well-posedness,  if additionally $P$ and $J$ are invertible. Proposition \ref{Propo: Characterisation of passive NBCOsystems} and Lemma \ref{sufficient condition Wellposedness NBCOs}  generalize this facts for non-autonomous boundary control and observation systems.
\begin{proposition}\label{Propo: Characterisation of passive NBCOsystems} The following assertion are equivalent. 
\begin{itemize}
\item [$(i)$]$\Sigma_N(\frA,\frB,\frC)$ is  (R,P,J)-scattering passive. 
\item [$(ii)$] For each $t\geq 0$ and all $(x,u)\in \V(t)$ we have 
\begin{equation}\label{characterization of passive system equation}
2\Re (\frA(t)x\mid P(t)x)_X\leq (R(t)u\mid \frB(t)x)_U-(\frC(t)x\mid J(t)\frC(t)x)_Y.
\end{equation}
\item [$(iii)$]   For each $t\geq 0$, the autonomous BCO-system $\Sigma(\frA(t),\frB(t),\frC(t))$ is $(R(t),P(t),J(t))$-scattering passive.
\end{itemize}
\end{proposition}
\begin{proof}
The equivalence of $(ii)$ and $(iii)$ has been  proved in Proposition \ref{Prop: Characterization of passive BCO system}. It remains to prove the equivalence of $(i)$ and $(ii)$. Assume that $(i)$ holds and let $s\geq 0$ and  let $(x_s, u_s)\in \mathcal V(s)$. Let $u: [s,\infty)\lra U$ such that $u(s)=u_s$. If $(x,y)$ is a classical solution of \eqref{NStateEqBCS}-\eqref{NObsEqBCS} corresponding to $(x_s,u)$ then $(x(t),u(t))\in \mathcal V(t), \ y(t)=\frC(t)x(t)$  and 
\begin{align}
\label{CharPassiveNA: proof Eq1}\frac{\d}{\d t} (P(t)x(t)\mid x(t))-(\dot P(t)x(t)\mid x(t))&=2\Re (\dot x(t)\mid P(t)x(t))
\\\label{CharPassiveNA: proof Eq2}&=2\Re (\frA(t)x(t)\mid P(t)x(t))
\end{align}
for all $t\geq s.$ Inserting this into \eqref{(R,P,J)-scattering passive Definition 1} yields
\begin{equation*}
2\Re (\frA(t)x(t)\mid P(t)x(t))\leq (R(t)u(t)\mid \frB(t)x(t))_U-(\frC(t)x(t)\mid J(t)\frC(t)x(t))_Y
\end{equation*}
for all $t\geq s$. The last inequality  $(ii)$  by taking $t=s$. Conversely, assume that $(ii)$ holds and let $(x,y)$ be a classical solution of \eqref{NStateEqBCS}-\eqref{NObsEqBCS}. Then $(x(t),u(t))\in \V(t)$ and  \eqref{CharPassiveNA: proof Eq1}-\eqref{CharPassiveNA: proof Eq2} holds for all $t\geq s.$ This together with \eqref{characterization of passive system equation} imply \eqref{(R,P,J)-scattering passive Definition 1}, which completes the proof.
\end{proof}

\begin{lemma}\label{sufficient condition Wellposedness NBCOs} Let $\Sigma_N(\frA,\frB,\frC)$ be  $(R,P,J)$-scattering passive such that $J\geq 0.$ Assume that  $P$ is strongly $C^1$-continuous and uniformly coercive with  
\begin{equation}\label{uniformly coerciveness of P}
 (P(t)x|x)\geq \beta\|x\|^2, \text{ for all } t\geq 0, x\in X
\end{equation}
for some constant $\beta>0.$ Then  each classical solution of \eqref{NStateEqBCS}-\eqref{NObsEqBCS} satisfies the following inequality 
\begin{equation}\label{balance inequality for NBCOs} \beta\|x(t)\|^2+\int_s^t(y(r)\mid J(r)y(r))\d r \leq c_{t,s} e^{\frac{1}{\beta}\int_s^t \|\dot P(r)\|\d r}\Big[\int_s^t(u(r)\mid R(r)u(r))\d r+\|x(s)\|^2\Big]\end{equation}
where $c_{t,s}=\max\{1,\underset{r\in [s,t]}{\max }\| P(r)\|\}$.
Therefore, $\Sigma_N(\frA,\frB,\frC)$ is well-posed provided that $J$  is uniformly coercive and $R\in L^\infty_{Loc}([0,\infty);\L(U))$.

\end{lemma}
\begin{proof} For the proof we follow a similar argument as in fourth steps of the proof of \cite[Theorem 4.1]{SchWei10}. Assume that $\Sigma_N(\frA,\frB,\frC)$ is $(R,P,J)$-scattering passive. Clearly  \eqref{(R,P,J)-scattering passive Definition 1} holds if and only if  \begin{align}\label{(R,P,J)-scattering passive Definition 2}
(P(t)x(t)\mid x(t))+\int_s^t(y(r)\mid J(r)y(r))\d r\leq &\int_s^t(u(r)\mid R(r)u(r))\d r\\\nonumber&+ \int_s^t(\dot P(r)x(r)\mid x(r))\d r+(P(s)x_s\mid x_s)\quad 
\end{align}
for all $t\geq s\geq 0$. Thus using \eqref{uniformly coerciveness of P} and that $J\geq 0$ we obtain 
	\begin{align*}
	\beta\|x(t)\|^2+\int_s^t(y(r)\mid J(r)y(r))\d r&
	\leq \int_s^t(u(r)\mid R(r)u(r))\d r+\|P(s)\|\|x(s)\|^2\\&\qquad+\int_s^t\|\dot P(r)\|\|x(r)\|^2\d r
\\ &\leq \int_s^t(u(r)\mid R(r)u(r))\d r+\|P(s)\|\|x(s)\|^2\\&\qquad+\int_s^t\frac{1}{\beta}\|\dot P(r)\|\Big[\beta \|x(r)\|^2+\int_s^r(y(\zeta)\mid J(\zeta)y(\zeta))\d \zeta\Big]\d r.
	\end{align*}
	Applying Gronwall's Lemma  yields 
	\begin{equation}\beta\|x(t)\|^2+\int_s^t(y(r)\mid J(r)y(r))\d r \leq  e^{\frac{2}{\beta}\int_s^t \|\dot P(r)\|\d r}\Big[\int_s^t(u(r)\mid R(r)u(r))\d r+\|P(s)\|\|x(s)\|^2\Big],\end{equation} which implies \eqref{balance inequality for NBCOs}. This completes the proof.
\end{proof}

\subsection{Multiplicative perturbed of NBCO-systems}\label{Section perturbation of NBCOS}
We will adopt the same notations of the previous sections.  
The  main purpose of this section is the study of some classes of NBCO-systems which are governed by a time-dependent multiplicative perturbation. More precisely, let $\Sigma_N(\frA,\frB,\frC)$ be a NBCO-system such that the boundary operators are constant, that is $\frC(t)=\frC$ and $\frB(t)=\frB$ for all $t\geq 0$. Thus the domain $\frA(t)$ should also be constant and we set $\D(\frA(t))=\frD$ for all $t\geq 0$.

 Further, throughout this section we assume that the  following assumption holds: 
\begin{assumption}\label{assumption 1 } 
\begin{enumerate}
\item  $M:[0,\infty)\lra \L(X)$and $R:[0,\infty)\lra \L(X)$ be two self-adjoint and uniformly coercive functions.
\item   $M(\cdot)x \in C^2([0,\infty);X)$ and $M^{-1}(\cdot)x \in C([0,\infty);X)$ for each $x\in X$.
\item $L(\cdot)x \in C^1([0,\infty);X)$  for each $x\in X$ such that $L$ and $M$ commute.
\end{enumerate}\end{assumption} 
For  each $t\geq 0$ we set
\begin{align*}
\frA_M(t):&=\frA(t)M(t)
\\ \frC_M(t):&=\frC M(t)\ \text{ and }\  \frB_M(t):=\frB M(t).
\end{align*}
We consider the following perturbed  system
\begin{align}
\label{NStateEqBCSPert2}\dot x(t)&=\frA_M(t)x(t), \quad x(0)=x_0,\\
\label{NConEqBCSPert2}\frB_M(t)x(t)&=u(t),\\
\label{NObsEqBCSPert2}\frC_M(t) x(t)&=y(t),
\end{align}
which we denote by $\Sigma_{N,M}(\frA, \frB,\frC)=\Sigma_{N}(\frA M, \frB M,\frC M)$. Let $\tilde B(t)$ be operators associated with $\Sigma_{M}(\frA, \frB,\frC)$ provided by Definition \ref{Def:BCOS}-$(ii).$
 Then $\tilde B_M(t):=M^{-1}(t)\tilde B(t)$ satisfies for each $t\geq 0$ all properties listed in Definition \ref{Def:BCOS}-$(ii)$. Moreover, the main operators associated with $\Sigma_{M,N}(\frA, \frB,\frC)$  are given by $\{A(t)M(t)\mid t\geq 0 \}$, where $D(A(t)M(t))=M^{-1}(t)(\frD\cap \ker(\frB))$ for each $t \geq 0$.

\medskip
\begin{lemma}\label{Lemma passive pert} The perturbed system  $\Sigma_{N,M}(\frA,\frB,\frC)$ is $(R,P,J)$-scattering passive  if and only if $\Sigma_{N}(\frA,\frB,\frC)$  is $(R,PM^{-1},J)$-scattering passive.
\end{lemma}
\begin{proof} For each $t\geq 0$ we set
\begin{equation*}\V_M(t):=
\left\{(x,u)\in X\times U\mid x\in D(\frA_M(t)) \text{ and } \frB_M(t)x=u\right\}.\end{equation*}
Then, $(x,u)\in \V_M(t)$ if and only if $(M(t)x,u)\in\V(t)$ for all $t\geq 0$. Assume now that  $\Sigma_{N}(\frA,\frB,\frC)$  is $(R,M^{-1}P,J)$-scattering passive and let $(x,u)\in \V_M(t)$. Using Proposition \ref{Propo: Characterisation of passive NBCOsystems} we obtain 
\begin{align*}
2\Re \big(\frA_M(t)x\mid &P(t)x\big)=2\Re \big(\frA(t)M(t)x\mid P(t)M^{-1}(t)M(t)x(t)\big)
\\&\leq (R(t)u\mid \frB M(t)x)_U-(\frC M(t)x(t)\mid J(t)\frC M(t)x(t))_Y
\\&=(R(t)u\mid \frB_M(t)x)_U-(\frC_M(t)x\mid J(t)\frC_M(t)x)_Y.
 \end{align*}
This implies, again by Proposition \ref{Propo: Characterisation of passive NBCOsystems}, that $\Sigma_{N,M}(\frA,\frB,\frC)$  is $(R,P,J)$-scattering passive. 
\par\noindent Conversely, assume that $\Sigma_{N,M}(\frA,\frB,\frC)$ is $(R,P,J)$-scattering passive. This means that $\Sigma_{N}(\frA M,\frB M,\frC M)$ is $(R,PM^{-1}M,J)$-scattering passive. Applying the first part of the proof yields that \[\Sigma_{N}(\frA,\frB ,\frC )=\Sigma_{N,M^{-1}}(\frA M,\frB M,\frC M)\] is $(R,PM^{-1},J)$-scattering passive. This completes the proof.
\end{proof}
In particular, the system $\Sigma_{N,M}(\frA,\frB,\frC)$ is $(R,M,J)$-scattering passive if and only if  the unperturbed system $\Sigma_{N}(\frA,\frB,\frC)$ is $(R,I,J)$-scattering passive.
According to the above assumptions, we remark that $(x,y)$ is a classical solution of  \eqref{NStateEqBCSPert2}-\eqref{NObsEqBCSPert2} if and only if $x$ is a classical solution of  \eqref{NStateEqBCSPert2}-\eqref{NConEqBCSPert2}. 
\medskip

Now we can formulate the first main result of this section.
\begin{theorem}\label{Main theorem pert} Assume that the following additional assumptions holds
\begin{itemize}
\item  [$(a)$] $\frA:[0,\infty)\lra \L(\frD,X)$ is strongly $C^1$-continuous.
\item [$(b)$] The main operators $A(t): \frD\cap \ker(\frB)\lra X$, $t\geq 0$ generate contraction $C_0$-semigroups. 
\item   [$(c)$] $\tilde{B}(\cdot)u \in C^2([0,\infty);U)$ for each $u\in U$.
\end{itemize}
Then the perturbed system $\Sigma_{N,M}(\frA,\frB,\frC)$ is a NBCO-system on $(X,U,Y)$. Furthermore, if we denote by $\mathcal W$ the associated evolution family, then for each $s\geq 0$ and $(x_s,u)\in X\times C^2([0,\infty);U)$ with $(M(s)x_s,u(s))\in \V(s)$ the system \eqref{NStateEqBCSPert2}-\eqref{NObsEqBCSPert2} has a unique classical solution $(x,y)$ given by
 \begin{align*}\label{VariContFormulaPer2}
x(t)&=W(t,s)x_s+\int_s^t W(t,r)\frA(r)\tilde B(r)u(r)\d r-\int_s^t W(t,r)\frac{\d}{\d r} \big[\tilde B_M(r)u(r)\big]\d r, \quad t\geq s,
\\ y(t)&=\frC_M(t)W(t,s)x_s+\frC_M(t)\int_s^t W(t,r)\frA(r)\tilde B(r)u(r)\d r-\frC_M(t)\int_s^t W(t,r)\frac{\d}{\d r} \big[\tilde B_M(r)u(r)\big]\d r,  \quad t\geq s.
\end{align*}
The system $\Sigma_{N,M}(\frA,\frB,\frC)$ is well-posed if in addition 
\begin{equation}\label{Eq Characterization of passive BCO systemb}
2\Re (\frA(t) x_0\mid  x_0)_X\leq (R(t)u_0\mid \frB x_0)_U-(\frC x_0\mid J(t) \frC x_0)_Y
\end{equation}
for all $t\geq 0$ and $(x_0,u_0)\in\V(t)$ where  $R=R^*\in L^\infty_{Loc}([0,\infty);\L(U))$ and  $J=J^*$ is uniformly coercive, where \begin{equation*}\V=
\left\{(x,u)\in X\times U\mid x\in D(\frA) \text{ and } \frB x=u\right\}.\end{equation*}
\end{theorem}
\begin{proof}
The  first and the second assertion follow from  Proposition \ref{Proposition Fattorini Trick Non-A}
and Corollary \ref{main corollary Abstract}, whereas the last assertion follows from Lemma \ref{Lemma passive pert}, Proposition \ref{Propo: Characterisation of passive NBCOsystems} and Lemma \ref{sufficient condition Wellposedness NBCOs}.
\end{proof}
Next we consider the case where $\frA(t)=L(t)\frA$ with $L(t)$ is as in Assumption \ref{assumption 1 } and such that $(\frA,\frB,\frC)$ is an autonomous BCO-system. This implies that $(L(t)\frA,\frB,\frC)$ is again an autonomous BCO-system for each $t\geq 0$ such that the associated operator $\tilde B$ is time-independent. In fact, if $\tilde B$ denotes the operator associated with the autonomous BCO-system $\Sigma(\frA, \frB,\frC)$, then it is easy to see that $\tilde B$ satisfies all properties listed in Definition \ref{Def:BCOS}-$(ii)$ corresponding to $(L(t)\frA,\frB,\frC).$
We consider the following perturbed  system
\begin{align}
\label{NStateEqBCSPertLeft}\dot x(t)&=L(t)\frA M(t)x(t), \quad x(0)=x_0,\\
\label{NConEqBCSPert2Left}\frB_M(t)x(t)&=u(t),\\
\label{NObsEqBCSPert2Left}\frC_M(t) x(t)&=y(t),
\end{align}
which we denote by $\Sigma_{N,M,L}(\frA, \frB,\frC)=\Sigma_{N}(L \frA M, \frB M,\frC M)$.
\medskip

 Clearly, the main operators associated with $\Sigma_{M,N,L}(\frA, \frB,\frC)$  are given by $\{L(t)AM(t)\mid t\geq 0 \}.$
\begin{theorem}\label{Main theorem pert2} Assume that the main operators $A: \frD\cap \ker(\frB)\lra X$ generate a contraction $C_0$-semigroup on $X.$ Then the perturbed system $\Sigma_{N,M,L}(\frA,\frB,\frC)$ is a NBCO-system on $(X,U,Y)$. Furthermore, if we denote by $\mathcal W$ the associated evolution family, then for each $s\geq 0$ and $(x_s,u)\in X\times C^2([0,\infty);U)$ with $(M(s)x_s,u(s))\in \V(s)$ the system \eqref{NStateEqBCSPertLeft}-\eqref{NObsEqBCSPert2Left} has a unique classical solution $(x,y)$ given by
 \begin{align*}\label{VariContFormulaPer2}
x(t)&=W(t,s)x_s+\int_s^t W(t,r)L(r)\frA \tilde Bu(r)\d r-\int_s^t W(t,r)\frac{\d}{\d r} \big[ M^{-1}(r)\tilde Bu(r)\big]\d r, \quad t\geq s,
\\ y(t)&=\frC_M(t)W(t,s)x_s+\frC_M(t)\int_s^t W(t,r)L(r)\frA \tilde Bu(r)\d r-\frC_M(t)\int_s^t W(t,r)\frac{\d}{\d r} \big[M^{-1}(r)\tilde Bu(r)\big]\d r,  \quad t\geq s.
\end{align*}
The system $\Sigma_{N,M,L}(\frA,\frB,\frC)$ is well-posed if in addition 
\begin{equation}\label{Eq Characterization of passive BCO systemb}
2\Re (L(t)\frA x_0\mid  x_0)_X\leq (R(t)u_0\mid \frB x_0)_U-(\frC x_0\mid J(t) \frC x_0)_Y
\end{equation}
for all $t\geq 0$ and $(x_0,u_0)\in\V$ where  $R=R^*\in L^\infty_{Loc}([0,\infty);\L(U))$ and  $J=J^*$ is uniformly coercive. 
\end{theorem}
\begin{proof}
The  first and the second assertion follow from  Proposition \ref{Proposition Fattorini Trick Non-A}
and Corollary \ref{main corollary Abstract2}, whereas the last assertion follows from Lemma \ref{Lemma passive pert}, Proposition \ref{Propo: Characterisation of passive NBCOsystems} and Lemma \ref{sufficient condition Wellposedness NBCOs}.
\end{proof}
\begin{remark} Theorem \ref{Main theorem pert2} is not a special case of Theorem \ref{Main theorem pert} since we do not assume that $P(t)A$ generates  a contractive $C_0$-semigroup on $X.$
\end{remark}
\section{Mild solutions for NBC-systems}\label{Section Mild solution}
As mentioned in Section \ref{BCOAutobonous}, for an autonomous BCO-system $\Sigma(\frA,\frB,\frC)$, for smooth input $u$ and initial data $x_0$,  the classical solution of the corresponding  boundary control system  can be formulated as 
\begin{equation}
x(t)=T(t)x_s+\int_{s}^t T_{-1}(t-s)Bu(s)ds, \qquad t\geq s.
\end{equation}

We recall that  $B\in\L(U,X_{-1})$ is given by \eqref{control operator for BOC}. If $x_s\in X$ and $u\in L^2([0,\infty);U)$, then  the above formula makes sense  and it is called the  mild solution in $X_{-1}$ of \eqref{StateEqBCS}-\eqref{ObsEqBCS}.  Moreover, it is well known that the mild solution belongs to $C([0,\infty);X)$ if $B$ is admissible for the semigroup $(T(t))_{t\geq 0},$ i.e., if for some $\tau>0$ one has
\[\int_{s}^\tau T_{-1}(\tau-s)Bu(s)ds\in X,\]
see, e.g., \cite[Proposition 4.2.4]{TucWei00}.

The main purpose of this section is to extend the conceps of mild solutions to non-autonomous boundary control and observation systems $\Sigma_N(\frA,\frB,\frC)$. In contrast to the  autonomous case, this is more delicate. In fact, firstly we remark that the extrapolation spaces $X_{-1,t}$ associated with the family $\{A(t)\mid t\geq 0\}$ of the main operators are in general time-dependent. Secondly,  in contrast to semigroups, it is not clear whether the evolution family $\U$ generated by $\{A(t)\mid t\geq 0\}$ can be extended to the extrapolation space even if the spaces  $X_{-1,t}$ are constant. However, if the latter condition holds, then we can still use the adjoint problem, i.e, $A^*(t)$, $t\geq 0$, and the associated backward evolution family to extend $\U$ to $\L(X_{-1})$. The idea to use a duality argument can be found  in \cite{ CheWei15, PauPoh12, SchWei10} to study some classes of non-autonomous systems.
\medskip

We will adopt here the notations of the previous sections. Let $\Sigma_N(\frA,\frB,\frC)$ be a NBCO-system. Then the main operators $\{A(t)\mid t\geq 0\}$ generate, by definition, an evolution family $\U=\{U(t,s)\mid (s,t)\in\Delta\}$ with regularity space $Y_{t}$, $t\geq 0$.
We restrict ourselves to case where $\{A(t)\mid t\geq 0\}$ have a common extrapolation space $X_{-1}$, i.e.,
\begin{equation}\label{MainAssExtSpace}X_{-1}:=X_{-1}(t)=X_{-1}(s) \text{ for all } t,s \geq 0.\end{equation}
According to \cite[Proposition 2.10.2]{TucWei00}, \eqref{MainAssExtSpace} holds if and only if $D(A^*(t))=D(A^*(s))$ for $t,s\in[0,\infty)$ and the corresponding graph norms are locally uniformly equivalent. In fact, $X_{-1}(t)$ is the dual space of $D(A^*(t))$ with respect to the pivot space $X$. This condition holds, if for instance $A(t)=AM(t)$ or $A(t)=A+M(t)$ and $M\in C^1([0,\infty);\L(X))$.
\par\noindent  In the following we denote  $\frD_*:=D(A^*(0))$ equipped with the graph norm and by $\langle\cdot,\cdot\rangle$ the duality between $X_{-1}$ and $\frD_*$. Recall from \eqref{control operator for BOC} 
\begin{equation}\label{control operator for BOCN} 
B(t)=\frA(t)\tilde B(t)-A_{-1}(t)\tilde B(t),\quad t\geq 0.
\end{equation}

\begin{proposition}\label{Extrapolation main Proposition} Assume that $\mathcal A^*:=\left\{A^*(t) \mid t\geq 0\right\}$ generates a backward evolution family $\U_*$. Then $U(t,s)$ has a unique  extension $U_{-1}(t,s)\in\L(X_{-1})$ for each $(t,s)\in \Delta$ and for each  $T>0$ there is $c_T>0$ such that  
\begin{equation}\label{Mild solution: LocUniBd}\sup_{0\leq s\leq t\leq T}\|U_{-1}(t,s)\|_{\L(X_{-1})}<c_T.\end{equation} Moreover, if  the assumptions of  Proposition \ref {Proposition Fattorini Trick Non-A} hold, then each classical solution $x$ of the boundary control system \eqref{NStateEqBCS}-\eqref{NConEqBCS} satisfies 
\begin{align}\label{variation of constant formula NBCO system} x(t)&=U(t,s)x_s+\int_s^tU_{-1}(t,r)B(r)u(r)\d r, \ t\geq s\geq 0.
\end{align}

\end{proposition}
\begin {proof} By  \cite[Proposition 2.9.3-(b)]{TucWei00} we obtain that for each $(t,s)\in \Delta$ the operator $U(t,s)$ has a unique  extension $U_{-1}(t,s)\in\L(X_{-1})$  since $[U(t,s)]^*\frD_*=U_*(t,s)\frD_*\subset \frD_*$.  Next, similar to the  proof of \cite[Proposition 2.7-(c)]{SchWei10} we show the uniform boundedness of $\U_{-1}$ on compact intervals. Next, we claim that for each $y\in \frD_*, x\in X_{-1}$ we have
\begin{equation}\label{proof main Thm ExtraEVF Eq1}\langle U_{-1}(t,s)x,y\rangle=\langle x,U_*(t,s)y\rangle \end{equation} In fact, this equality holds for $x\in X$ by Lemma \ref{Lemma Back-ForwardEVFRelation}-$(ii)$ since   \[\langle x,U_*(t,s)y\rangle=(x\mid U_*(t,s)y)=(U(t,s)x\mid y)=\langle U_{-1}(t,s)x,y\rangle.\]
 Remark that $U_*(t,s)y\in \frD_*$, thus the claim follows since  $X$ is dense  in $X_{-1}$.
\par\noindent Using again Lemma \ref{Lemma Back-ForwardEVFRelation} and \eqref{proof main Thm ExtraEVF Eq1}, we obtain for each $y\in\frD_*$  
\begin{align*}
\frac{\d}{\d s}\langle U(t,s)\tilde B(s)u(s),y\rangle &=\frac{\d}{\d s}(\tilde B(s)u(s),U_*(t,s)y)
\\&=(\frac{\d}{\d s}[\tilde B(s)u(s)],U_*(t,s)y)-(\tilde B(s)u(s),A^*(s)U_*(t,s)y)
\\&=(\frac{\d}{\d s}[\tilde B(s)u(s)],U_*(t,s)y)-\langle U_{-1}(t,s)A_{-1}(s)\tilde B(s)u(s),y\rangle
\\&=(U(t,s)\frac{\d}{\d s}[\tilde B(s)u(s)],y)-\langle U_{-1}(t,s)A_{-1}(s)\tilde B(s)u(s),y\rangle.
\end{align*}
Integrating over $[s,t]$, we obtain 
\begin{equation} 
\int_s^t U_{-1}(t,r)A_{-1}(r)\tilde B(r)u(r)\d r=-\tilde B(t)u(t)+U(t,s)\tilde B(s)u(s)+\int_s^t U(t,r)\frac{\d}{\d r}[\tilde B(r)u(r)]\d r.
\end{equation}
Inserting this equality in \eqref{classSolNBOC}, we obtain that a classical solution $x$ of \eqref{NStateEqBCS}-\eqref{NConEqBCS} satisfies \eqref{variation of constant formula NBCO system}.
\end{proof}
If the assumptions of Proposition \ref{Extrapolation main Proposition} hold, then for $x_s\in X$ and $u\in L^2_{Loc}([0,\infty);U)$ we see that \eqref{variation of constant formula NBCO system} is well defined with value in $X_{-1}$ provided $B(\cdot)u(\cdot)\in L^1_{Loc}([0,\infty);X_{-1})$. In fact, \eqref{Mild solution: LocUniBd} guaranties that the integral term on the right hand side of \eqref{Definition mild solution of NBCO system} is well defined.
Thus the following definition makes sense. 
\begin{definition} Let $\Sigma_N(\frA,\frB,\frC)$ be a NBCO-system and let $\U$ and  $\{B(t)\mid t\geq 0\}$ the associated evolution family and control operators, respectively. Let $(x_s,u)\in X\times L^2_{Loc}([0,\infty);U)$. If  $U(t,s)$ has a unique  extension $U_{-1}(t,s)\in\L(X_{-1})$ for each $(t,s)\in \Delta$ such that $U_{-1}(t,\cdot)B(\cdot)u(\cdot) \in L^1_{Loc}([0,\infty);X_{-1})$, then the function  \begin{align}\label{Definition mild solution of NBCO system} x(t)&=U(t,s)x_s+\int_s^tU_{-1}(t,r)B(r)u(r)\d r, \ t\geq s\geq 0,
\end{align}
 is called the \textit{mild solution  of  \eqref{NStateEqBCS}-\eqref{NObsEqBCS} in $X_{-1}$}. Further, \eqref{Definition mild solution of NBCO system}  is called  a  \textit{mild solution  of  \eqref{NStateEqBCS}-\eqref{NObsEqBCS} (in $X$)}, if in addition \begin{align}\label{Def Mild solution in X}\Phi_{t,s}u:= \int_s^tU_{-1}(t,r)B(r)u(r)\d r\in X, \ \text{for all} \  (t,s)\in \Delta,\end{align}
 and $x\in C([s,\infty);X)$.
\end{definition}
This definition is related to the notion of \textit{admissibility} for non-autonomous linear systems. More precisely, recall that a family $\{B(t)\mid t\geq 0\}\subset \L(U,X_{-1})$ is $L^2$-admissible for a given  evolution family $\U$ that admit an extension to $\L(X_{-1})$ if $U_{-1}(t,\cdot)B(\cdot)u(\cdot) \in L^1_{Loc}([0,\infty);X_{-1})$, \eqref{Def Mild solution in X} holds and for each $T>0$ there exists $c_T>0$ such that 
\begin{equation}
\Big\|\int_s^tU_{-1}(t,r)B(r)u(r)\d r\ \Big\|_X^2\leq c_T \int_s^t\|u(r)\|_U^2\d r
\end{equation}
for each $u\in L^2_{Loc}([0,\infty);U)$ and all $0\leq s\leq t\leq T$ \cite[Definition 3.3]{Sch02}. For $L^2$-admissible control operators we have that $(t,s)\mapsto\Phi_{t,s}u$ is continuous on $\Delta$ with value in $X$ \cite[Proposition 3.5-(2)]{Sch02}.
\medskip

\begin{proposition}\label{Proposition 1 Mild solution} Assume that  $\mathcal A^*:=\left\{A^*(t) \mid t\geq 0\right\}$ belongs to the Kato-class and $\{B(t)\mid t\geq 0\}$ is $L^2$-admissible. Then for each $(x_s,u)\in X\times L^2_{Loc}([0,\infty);U)$ with $B(\cdot)u(\cdot)\in L^1_{Loc}([0,\infty);X_{-1})$ 
the system \eqref{NStateEqBCS}-\eqref{NObsEqBCS} has a unique mild solution in $X$.
\end{proposition}
\begin{proof} The proof follows from Lemma \ref{Lemma Back-ForwardEVFRelation}-$(i)$ and Proposition \ref{Extrapolation main Proposition}.
\end{proof}
If $\Sigma_N(\frA,\frB,\frC)$ is a well-posed NBCO-system  and the classical solutions is given by \eqref{variation of constant formula NBCO system},  then the corresponding family $\{B(t)\mid t\geq 0\}$ is $L^2$-admissible provided \[U_{-1}(t,\cdot)B(\cdot) L^2_{Loc}([0,\infty);U)\subset L^1_{Loc}([0,\infty);X_{-1}).\]
Thus the following corollary follows from Proposition \ref{Proposition 1 Mild solution}, Lemma \ref{sufficient condition Wellposedness NBCOs} and \eqref{Mild solution: LocUniBd}.
\begin{corollary}\label{Corollary 1 Mild solution} Assume $\Sigma_N(\frA,\frB,\frC)$ is $(R,P,J)$-scattering passive such that $R\in L^\infty_{Loc}([0,\infty);\L(U))$ and  $J, P$ are uniformly coercive. In addition, we assume that  $\mathcal A^*:=\left\{A^*(t) \mid t\geq 0\right\}$ belongs to the Kato-class. Then for each $(x_s,u)\in X\times L^2_{Loc}([0,\infty);U)$ with $B(\cdot)u(\cdot)\in L^1_{Loc}([0,\infty);X_{-1})$ 
the system \eqref{NStateEqBCS}-\eqref{NObsEqBCS} has a unique mild solution in $X$.
\end{corollary}
\medskip
Finally,  if Assumption \ref{assumption 1 } holds such that $A(t)$ generates contractive $C_0$-semigroup on $X$ for each $t\geq 0$ then we can follow   \cite[Section 2, page 8]{SchWei10} to deduce that the extrapolation spaces corresponding to $A(t)M(t)$, $t\geq 0$ can be all identified with $X_{-1}$ and that $[A(t)M(t)]_{-1}=A_{-1}(t)M(t)$  for every $t\geq 0$. 

\begin{corollary}\label{Corollary 2 Mild solution}  Assume that Assumption \ref{assumption 1 } holds such that the adjoint operators $\left\{A(t)^*\mid t\geq 0\right\}$ have a common domain. Then the perturbed system \eqref{NStateEqBCSPert2}-\eqref{NConEqBCSPert2} has a unique mild solution in $X$ if the unperturbed system $\Sigma_{N,\id}(\frA,\frB,\frC)$ is $(R,I,J)$-scattering passive with $R\in L^\infty_{Loc}([0,\infty);\L(U))$ and  $J, P$ are uniformly coercive.
\end{corollary}
\begin{proof} The proof is an easy consequence of  Corollary \ref{Corollary 1 Mild solution} and Lemma \ref{Lemma passive pert}.
\end{proof}
\section{Application to non-autonomous Port-Hamiltonian systems}\label{sec6}
Let $N, n\in\N$ be fixed and let $X:=L^2([a,b];\K^{n})$ where $\K=\R$ or $\C$. In this section we investigate the well-posedness  of the  linear non-autonomous \textit{port-Hamiltonian systems of order $N\in\N$}, given by the boundary control and observation system
\begin{align}
\label {port-Hamiltonian system} \frac{\partial}{\partial t}x(t,\zeta)&=\sum_{k=1}^N P_k(t) \frac{\partial^k }{\partial \zeta^k}\big[\H(t,\zeta)x(t,\zeta)\big]+P_0(t,\zeta)\H(t,\zeta)x(t,\zeta),\quad t\geq 0,\ \zeta\in(a,b)
\\ \H(0,\zeta)x(0,\zeta)&=x_0(\zeta), \qquad\qquad \zeta\in(a,b),
\\  \label{port-Hamiltonian systemBC}u(t)&=W_{B,1} \tau(\H x)(t),\qquad \qquad t\geq 0, 
\\ \label{port-Hamiltonian systemBC2}0&= W_{B,2} \tau(\H x)(t),\qquad \qquad t\geq 0,
\\ \label{Boundary Observation}y(t)&=W_C \tau(\H x)(t),\quad\quad\quad \quad  t\geq 0.
\end{align}
 Here $\tau$ denotes the  \textit{trace operator} $\tau: H^N ((a,b);\K^n)\lra \K^{2Nn}$ defined by
 \[ \tau(x):=\big( x(b),x'(b),\cdots, x^{N-1}(b),x(a),x'(a)\cdots, x^{N-1}(a)\big),\]
  $P_k(t)$ is $n\times n$ matrix for all $t\geq 0$, $k=0,1,\cdots,N$, $\H(t,\zeta)\in\K^{n\times n}$ for all $t\geq 0$ and almost every $\zeta\in[a,b]$, $W_{B,1}$ is a $m\times 2nN$-matrix, $W_{B,2}$ is $(nN-m)\times 2nN$-matrix and $W_C$ is a $d\times 2nN$-matrix. Finally, $u(t)\in U:=\K^{m}$ denotes the input and $y(t)\in Y:=\K^{d}$ is the output at time $t$.
\medskip

Set $W_B:=\left[\begin{smallmatrix}
W_{B,1}\\
W_{B,2} 
\end{smallmatrix}\right]$, $\Sigma:=
\left[\begin{smallmatrix}
	0& I\\
	I&0 
\end{smallmatrix}\right]$ and for each $t\geq 0$ we set 
\[
Q(t):=\begin{bmatrix} P_1(t)&P_2(t)&\cdots& &P_N(t) \\ 
-P_2(t)&-P_3(t)&\cdots&-P_N(t)&0\\
\vdots&&\adots&\adots&\vdots\\
 \vdots&\adots& \adots&& \vdots\\
(-1)^{N-1}P_N(t)&0&\cdots&\cdots&0&
\end{bmatrix}.\] $R_{ext}(t):=\left[\begin{smallmatrix} Q(t) &-Q(t)\\ \I& \I\end{smallmatrix}\right]$ and  $W_B(t):=W_B R_{ext}^{-1}(t), W_C(t):=W_B R_{ext}^{-1}(t)$.

In this section we  assume the following assumptions:
 \begin{assumption}\label{Assumption NpHs}{\hspace{2ex}}
\begin{itemize} 
\item $W_B$ has full rank and $W_B(t)\Sigma W_B^*(t)\geq 0$
for all $t\geq 0$.
\item $P_N(t)$ is invertible and 
$P^*_k(t)=(-1)^{k-1}P_k(t)$ for  all $k\geq 1$,  $t\geq 0$,
\item $P_k\in C^1([0,\infty);L^\infty(a,b;\C^{n\times n}))$ for all $t\geq 0$ and $k=0,1,\cdots N$.
\item  $\H\in C^2([0,\infty);L^\infty([a,b];\C^{n\times n}))$ and there exist $m, M\geq 0$ such that 
\begin{equation*}\label{coercivity H}
m \leq \H(t,\xi)=\H^*(t,\xi)\leq M, \quad {\rm a.e. }\ \xi\in [a,b], t\geq 0.
\end{equation*}
\end{itemize}
\end{assumption}
\noindent
Under these assumptions, the port-Hamiltonian system (\ref{port-Hamiltonian system})-(\ref{Boundary Observation}) can be written as a non-autonomous boundary control and observation system in the sense of Definition \ref{Def: NBCO}-$(ii)$. In fact, on the  Hilbert space $X$ we consider the (maximal) \textit{port-Hamiltonian operators} 
\begin{equation}
\frA(t) x=\sum_{k=0}^N P_k(t) \frac{\partial^k }{\partial \zeta^k}x
\quad \text{ with domain } \quad D(\frA(t))=\left \{H^N([a,b];\K^n)\mid W_{B,2}  \tau(x)=0\right\}
\end{equation}Then $(\frA(t), D(\frA(t)))$ is a closed and densely defined operator and its graph norm $\|\cdot\|_{D(\frA(t))}$ is equivalent to the Sobolev norm $\|\cdot\|_{H^N((a,b);\K^n)}$ as $P_N(t)$ is invertible. Moreover, for each $t\geq 0$ the operator $A(t): D(A(t))\subset X\lra X$ defined by 
\begin{align}
\label{NphsOpDef} A(t)x&=\frA(t) x\quad x\in D(A(t)) \\
\label{NphsOpeDomain} D(A(t))&=\Big\{x\in H^N((a,b);\K^n)\mid W_B \tau(x)=0 \Big\}
\end{align}
generates a contractive $C_0$-semigroup on $X$. Further, we define the input operator $\frB$ and output operator $\frC$ a follows 
\begin{align*} \mathfrak B: H^N((a,b);\K^n)&\lra U,\  \frB x:= W_{B,1}  \tau(x),
\end{align*}
and 
 \begin{align*} \mathfrak C: H^N((a,b);\K^n)&\lra Y,\  \frC x:= W_C  \tau(x).
 \end{align*} 
The operator $\frC$ is a linear and bounded operator from $D(\frA(t))$ to $Y$, since the trace operator $\tau$ is bounded and the norm graph norm of $D(\frA)$ is equivalent to the $H^N((a,b);\K^n)$-norm. Moreover, Lemma \ref{Lemma operator B tilde PHS} below shows that  there exists an operator $\tilde B\in \L(U,X)$ which is independent of $t\geq 0$ satisfying the assumption $(ii)$ of Definition \ref{Def:BCOS}. The proof of this fact follows by a minor modification of the proof of \cite[Theorem 11.3.2]{Jac-Zwa12} and that of \cite[Lemma 3.2.19]{Au16} (see also the second step of the proof of  \cite[Theorem 4.2]{LZM05}). 
\begin{lemma}\label{Lemma operator B tilde PHS} There exists a linear operator $\tilde B\in\L(\K^m,X)$ such that $\tilde B \K^m \subset D(\frA(t)),$  $\frA(t)\tilde B\in\L(\K^m,X)$ for each $t\geq 0$ and $\frB\tilde B=\I_{\K^m}=I_U.$
\end{lemma}
\begin{proof}
Since the $nN\times 2nN$-matrix $W_B$ has full rank $nN$ there exists a $2nN\times nN$-matrix $S$ such that 
\begin{equation}\label{Lemma operator B tilde proof PHS Eq1}
W_BS=\begin{bmatrix}
W_{B,1}\\
W_{B,2} 
\end{bmatrix}S=\begin{bmatrix}
I_{\K^m}& 0\\
0&0 
\end{bmatrix}.
\end{equation}
In fact, one can choose $S$ as follows
\[S=W_B^*(W_BW_B^*)^{-1}\begin{bmatrix}
I_{\K^m}& 0\\
0&0 
\end{bmatrix}.\]
Let us write $S=\left[\begin{smallmatrix} S_{11}&S_{12}\\ S_{21}&S_{22}\\ \cdot &\cdot\\ \cdot &\cdot\\ \cdot &\cdot\\ S_{(2nN)1}&S_{(2nN)2}\end{smallmatrix}\right]=:\left[\begin{smallmatrix}\tilde S_1&\tilde S_2\end{smallmatrix}\right],$ where $S_{j1}, j=1,\cdots, 2nN,$ are $1\times m$ matrices. 
\par 
Next, let $\{e_j\}_{j=1}^{2Nn}$ be the standard orthogonal basis in $\K^{2nN}.$ For each $j=1,2,\cdots 2nN$ we take $f_j\in H^{N}(a,b;\K^n)$ such that $\tau(f_j)=e_j$ \cite[Lemma A.3]{Vi-Go-Zw-Sc05}, and we define the operator $\tilde B\in\L(\K^m,X)$ by 
\begin{equation} \tilde Bu:=\sum_{j=1}^{2nN}S_{j1}uf_j\quad u\in \K^m.\end{equation}
Thus $B\in \L(\K^m, H^N((a,b);\K^n)).$ Furthermore, \eqref{Lemma operator B tilde proof PHS Eq1} implies that $W_{B,2}\tilde S_1=0$ and thus 
\begin{align*}W_{B,2}\tau(\tilde Bu)&=W_{B,2}\sum_{j=1}^{2nN}S_{j1}u\tau(f_j)\\&=W_{B,2}\sum_{j=1}^{2nN}S_{j1}u e_j=W_{B,2}\tilde S_1 u=0
\end{align*}
for every $u\in\K^m.$ We deduce that $\tilde B \K^m \subset D(\frA(t))$ for all $t\geq 0.$ It follows that $\Sigma(\frA(t),\frB,\frC)$ is for each $t\geq0$ a BCO-system on $(L^2([a,b];\K^n),\K^m,\K^d)$. Using \eqref{Lemma operator B tilde proof PHS Eq1} once more, we obtain that \[\frB\tilde B u=W_{B,1}\tau(\tilde B u)=W_{B,1}\tilde S_1u=u\] for all $u\in \K^m.$ This completes the proof.
\end{proof}
\medskip
Moreover, if in addition the following assumption holds
 \begin{assumption}\label{Assumption NpHsPassivity}{\hspace{2ex}}
\begin{itemize} 
\item $nN=m=d$ (and thus $W_B=W_{B,1}$ or equivalently $W_{B,2}=0$),
\item $R=\{R(t)\mid t\geq  0 \}$ and $J=\{J(t)\mid t\geq 0\}$ are bounded and self adjoint operators on $\K^n$,
\item $\Re P_0(t,\zeta) \leq 0$ for all $t\geq 0$ and a.e. $\zeta\in[a,b]$,
\item the matrix  $W_{C}$ has  full rank,
\end{itemize}
\end{assumption}
then we obtain: 
\begin{lemma} Under assumptions \ref{Assumption NpHsPassivity} and \ref{Assumption NpHs}
for each $t\geq 0$ the autonomous port-Hamiltonian system $\Sigma(\frA(t),\frB,\frC)$ is $(R(t),\I,J(t))$-scattering passive if 
 \begin{equation}\label{Scattering passive-CondPHS}P_{W_B,W_C}(t):=\Big( \begin{bmatrix} W_B(t)\\  W_C(t)\end{bmatrix}\Sigma  
\begin{bmatrix} W_B^*(t)& W_C^*(t)\end{bmatrix}\Big)^{-1}\leq \begin{bmatrix} 2R(t) &0\\ 0&-2J(t) \end{bmatrix}.
\end{equation}
\end{lemma}
\begin{proof}
Using \cite[Lemma 3.2.13]{Au16} we obtain 
\begin{equation}
\label{Scattering passive Eq0}\Re (\frA(t)x\mid x)=\Re (R_{ext}(t)\tau(x)\mid \Sigma R_{ext}(t)\tau(x))+\Re (P_0(t)x\mid x).
\end{equation}
Inserting  
\[\begin{bmatrix}W_B\tau (x)\\W_C\tau(x)\end{bmatrix}=\begin{bmatrix}W_B\\W_C\end{bmatrix}R_{ext}^{-1}(t)R_{ext}(t)\tau(x)=\begin{bmatrix}W_B(t)\\W_C(t)\end{bmatrix}R_{ext}(t)\tau(x)\]
into \eqref{Scattering passive Eq0} we obtain that
\begin{align}
2\Re(\frA(t)x\mid x)\leq &\Big<\begin{bmatrix}W_B  \tau(x)\\ W_C  \tau(x)\end{bmatrix}\Big|\begin{bmatrix} W_B(t)^*& W_C(t)^*\end{bmatrix}^{-1}\Sigma \begin{bmatrix} W_B(t)\\  W_C(t)\end{bmatrix}^{-1}\begin{bmatrix}W_B  \tau(x)\\ W_C  \tau(x)\end{bmatrix}\Big >_{\K^{2nN}}\quad t\geq 0,
\end{align}
holds for every $x\in H^{N}([a,b];\K^n)$, since $Re P_0(t,\zeta)\leq0$. Now the claim follows by Lemma \ref{Prop: Characterization of passive BCO system}.
\end{proof}

\par\noindent Finally, the assumption on $\H$ ensures that the family of operators $M(t) := \H(t)(\cdot):=\H(t,\cdot)$ as matrix multiplication operators on $L^2(a,b;\K^n)$ satisfies all assumptions of Section \ref{Section perturbation of NBCOS}. 
\medskip

Our abstract results in the previous sections hence yield the following  main result.
\medskip
\begin{theorem}\label{main result Phs1}If Assumption \ref{Assumption NpHs} holds, then the port-Hamiltonian system \eqref{port-Hamiltonian system}-\eqref{Boundary Observation} is a non-autonomous boundary control and observation system. Furthermore, there exists a unique evolution family $\W$ in $L^2([a,b];\K^n)$ such that for each  $x_0\in H^N((a,b);\K^n)$ and $u\in C^2([a,b];\K^m)$ with, $W_{B,1} \tau(x_0)=u(0)$ and $W_{B,2} \tau({x_0})=0$ we have  \begin{align*}\label{VariContFormulaPHS}
x(t)=W(t,0){\H^{-1}(0,\zeta)}x_0+\int_0^t W(t,r)\frA(r)\tilde Bu(r)\d r+&\int_0^t W(t,r)\H^{-1}(r)\dot\H(r)\H^{-1}(r)\tilde Bu(r)\d r \\&
\qquad\qquad-\int_0^t W(t,r)\H^{-1}(r)\tilde B\dot{u}(r)\d r,\quad t\geq 0,
\end{align*}
\begin{align*}
y(t)=\frC \H(t)W(t,0){\H^{-1}(0,\zeta)}x_0+&\frC \int_0^t \H(t)W(t,r)\Big[\frA(r)-\H^{-1}(r)\dot\H(r)\H^{-1}(r)\Big]\tilde Bu(r)\d r
\\&-\frC\int_0^t \H(r)W(t,r)\H^{-1}(r)\tilde B\dot{u}(r)\d r, \quad t\geq 0.
\end{align*}
is the unique classical solution of \eqref{port-Hamiltonian system}-\eqref{Boundary Observation}. If in addition Assumption \ref{Assumption NpHsPassivity} and \eqref{Scattering passive-CondPHS} hold, then  \eqref{port-Hamiltonian system}-\eqref{Boundary Observation} is $(R,\H,J)$-scattering passive and the classical solution $(x,y)$ satisfies the balance inequality 
\begin{equation}\label{balance inequality for NPHs}m\|x(t)\|^2+\int_s^t(y(r)\mid J(r)y(r))\d r \leq c_{t,s} e^{\frac{1}{m}\int_s^t \|\dot \H(r)\|\d r}\Big[\int_s^t(u(r)\mid R(r)u(r))\d r+\|x(s)\|^2\Big]\end{equation}
where $c_{t,s}=\max\{1,\underset{r\in [s,t]}{\max }\| \H(r)\|\}$. Moreover,  \eqref{port-Hamiltonian system}-\eqref{Boundary Observation} is well posed if in addition $J$ is uniformly coercive and $R\in L_{Loc}^\infty([0,\infty);\L(\K^n))$.
\end{theorem}
\par\noindent 

Finally, we  give a result on the existence of mild solution of the non-autonomous port-Hamiltonian system. For that we assume that  $nN=m=d$. Then it is known \cite[Lemma A1]{LZM05} (see also \cite[Section~7.3]{Jac-Zwa12}) that there exist  a matrix $V\in \K^{nN\times nN}$ and an invertible matrix $S\in \K^{nN\times nN}$ such that 
\[W_B=S\begin{bmatrix} I+V& I-V\end{bmatrix}\]
with $VV^*\geq \I$. Further, we have $\ker W_B=\ran 
\left[\begin{smallmatrix} I-V\\ -I-V\end{smallmatrix}\right]$ . For each $t\geq 0$, the adjoint operator $A^*(t):D(A^*(t))\lra X$ of \eqref{NphsOpDef}-\eqref{NphsOpeDomain} is given  by 
\begin{align}
\label{NphsOpDefad} A^*(t)x&=-\frA(t) x\quad x\in D(A^*(t)) \\
\label{NphsOpeDomainad} D(A^*(t))&=\Big\{x\in H^N(a,b;\K^n)\mid  \begin{bmatrix} I-V^*&-I-V^*\end{bmatrix}\begin{bmatrix}Q(t)&0\\0&-Q(t)\end{bmatrix}\tau(x)=0 \Big\}
\end{align}
see e.g., \cite[Theorem 2.24]{Vi07}, \cite[Proposition 3.4.3]{Au16}. We deduce that the domain of $A^*(t)$ are time-independent if for instance all matrices $P_k, k=1,2,\cdots N$ are constant. Thus using Corollary \ref{Corollary 2 Mild solution} we obtain the followin proposition.
\begin{proposition}\label{main result Phs2} Assume that Assumption \ref{Assumption NpHs} and Assumption \ref{Assumption NpHsPassivity} hold with $P_k, k=1,2,\cdots N$ are constant and $J$ is uniformly coercive and $R\in L_{Loc}^\infty([0,\infty);\L(\K^n))$. If  \eqref{Scattering passive-CondPHS} holds, then the non-autonomous  system \eqref{port-Hamiltonian system}-\eqref{port-Hamiltonian systemBC2} has a unique mild solution. 
\end{proposition}

We closed this section by  some examples of physical systems which can be modelled as a non-autonomous port-Hamiltonian system. Then the existence of classical and mild solutions as well as well-posedness can be checked by a simple application of the  abstracts results presented in this section. Here we will present just two relevant examples, however various other control systems fit into the framework of port-Hamiltonian system and  into the general class of NBCO-systems. 
\subsection{Vibrating string}\label{Ex vibrating string} Let us consider the model of vibrating string on the compact interval $[a,b]$. The string is fixed at the left end point $a$ and at the right end point $b$ a damper is attached. The Young's modulus and the mass density of the string are assumed to be time- and spatial dependent. Let us denote by $\omega(t,\zeta)$ the vertical position of the string at position $\zeta\in[a,b]$ and time $t\geq 0$. Then the evolution of the controlled vibrating string can be modelled by a non-autonomous wave equation of the form

\begin{align}
\label{Eq1: State eq}\frac{\partial  }{\partial t}\Big(\alpha(t)\rho(t,\zeta)\frac{\partial w }{\partial t}(t,\zeta) \Big)&=\frac{1}{\alpha(t)}\frac{\partial}{\partial \zeta}\Big(T(t,\zeta)\frac{\partial w}{\partial \zeta}(t,\zeta)\Big), \ \zeta\in[a,b],\ t\geq 0,\\
\label{Eq2: BCd}T(b,t)\frac{\partial w}{\partial \zeta}(t,b)+k\alpha(t)\frac{\partial w}{\partial t}(t,b)&=u_1(t),
\\\label{Eq1: BCd}\frac{\partial w}{\partial t}(t,a)&=u_2(t).
\end{align} 
We assume that $k\geq 0$ and $T, \rho \in C^2([0,\infty);L^\infty(a,b)) \cap C_b([0,\infty);L^\infty(a,b))$ such that  for some $m>0$, for a.e\  $\zeta\in[a,b]$ and all $t\geq 0$ we have $m^{-1} \leq \rho(t,\zeta), T(t,\zeta)\leq m$, moreover,  $\alpha \in C^1([0,\infty))$  is strictly positive. We take as state variable the momentum-strain couple $x:=(\alpha\rho \frac{\partial w}{\partial t}, \frac{\partial w}{\partial \zeta})$. 
Then the first equation can be equivalently written as follows
\begin{equation}
\frac{\partial}{\partial t}x(t,\zeta)=\frA(t) \H(t,\zeta)x(t,\zeta)
\end{equation}
where \[\frA(t):=\begin{bmatrix} 
0&1/\alpha(t)\\
 1/\alpha(t)&0\end{bmatrix}\frac{\partial}{\partial \zeta}\ \ \text{ and }\H(t,\zeta):=\begin{bmatrix} 
 \frac{1}{\rho(t,\zeta)}&0\\
 0&T(t,\zeta)\end{bmatrix}. \]
 Indeed, we have
\begin{align*}
\frA(t) \H(t,\zeta)x(t,\zeta)&=\begin{bmatrix} 0 &
 1/\alpha(t)\\
 1/\alpha(t)&0\end{bmatrix}\frac{\partial}{\partial \zeta}\begin{bmatrix} 
 \frac{1}{\rho(t,\zeta)}&0\\
 0&T(t,\zeta)\end{bmatrix}\begin{bmatrix}\alpha(t)\rho(t,\zeta) \frac{\partial w}{\partial t}(t,\zeta)\\ \frac{\partial w}{\partial \zeta}(t,\zeta)\end{bmatrix}
\\&=\begin{bmatrix} 0 &
 1/\alpha(t)\\
 1/\alpha(t)&0\end{bmatrix}\frac{\partial}{\partial \zeta}\begin{bmatrix}\alpha(t) \frac{\partial w}{\partial t}(t,\zeta)\\ T(t,\zeta)\frac{\partial w}{\partial \zeta}(t,\zeta)\end{bmatrix}
\\&=\begin{bmatrix}1/\alpha(t)\frac{\partial}{\partial \zeta}\big(T(t,\zeta)\frac{\partial w}{\partial \zeta}(t,\zeta)\big)\\ \frac{\partial}{\partial \zeta} \frac{\partial w}{\partial t}(t,\zeta) \end{bmatrix}
\\&=\begin{bmatrix}\frac{\partial  }{\partial t}\Big(\alpha(t)\rho(t,\zeta)\frac{\partial w }{\partial t}(t,\zeta) \Big)\\ \frac{\partial}{\partial t} \frac{\partial w}{\partial \zeta}(t,\zeta) \end{bmatrix}=\frac{\partial}{\partial t}x(t,\zeta).
\\
\end{align*}
Moreover, the boundary conditions \eqref{Eq1: BCd}-\eqref{Eq2: BCd} with $u=(u_1,u_2)=0$ can be equivalently written as follows 
\[W_B\begin{bmatrix}\H(t,b)x(t,b)\\\H(t,a)x(t,a)\end{bmatrix}:=
 \begin{bmatrix} k&1&0&0\\
 	0&0&1&0\end{bmatrix}\begin{bmatrix}\H(t,b)x(t,b)\\\H(t,a)x(t,a)\end{bmatrix}=\begin{bmatrix}0\\ 0\end{bmatrix}\]
The $2\times 4$ matrix $W_B$ has full rank. Next,   \[W_B(t)=W_B \begin{bmatrix}
0&\alpha(t)& 1&0\\
\alpha(t)&0& 0&1\\
0&-\alpha(t)& 1&0\\
-\alpha(t)&0& 0&1\\
\end{bmatrix}=
\begin{bmatrix} \alpha(t)&k\alpha(t)&k&1\\
0&-\alpha(t)&1&0\end{bmatrix}\]
and  $W_B(t)\Sigma W_B^*(t)=\begin{bmatrix}
4k\alpha(t)& 0\\
0&0
\end{bmatrix}\geq 0$.
The corresponding matrices $W_{B,1}, W_{B,2}$ and the  corresponding boundary operator $\frB$ can be defined as follows:
 
\textbf{Case $u_2=0:$} $W_{B,2}= \begin{bmatrix}
 	0&0&1&0\end{bmatrix}$ and 
\begin{align*} \mathfrak B: &H^N((a,b);\K^2)\lra U=\K,\  \\\frB x:&=W_{B,1}\tau(x):= \begin{bmatrix} k&1&0&0
 \end{bmatrix}\begin{bmatrix}x(b)\\x(a)\end{bmatrix}.
\end{align*}

\textbf{Case $u_2\neq 0:$} $W_{B,2}= 0$ and
\begin{align*} \mathfrak B: H^N((a,b);\K^2)&\lra U=\K^2,\  \frB x:= \begin{bmatrix} k&1&0&0\\
 	0&0&1&0\end{bmatrix}\begin{bmatrix}x(b)\\ x(a)\end{bmatrix}
\end{align*}
 For each $S,V\in \K^{2\times 2}$ such that  $S$ is invertible and $VV^*\geq I$ we can  we take 
\begin{equation}\label{output Vibrating}y(t)= S\begin{bmatrix} \I+ V& \I- V\end{bmatrix}\begin{bmatrix}\H(t,b)x(t,b)\\\H(t,a)x(t,a)\end{bmatrix}\end{equation}
as an output of \eqref{Eq1: BCd}-\eqref{Eq2: BCd}. Thus, we  are in the position to apply Theorem \ref{main result Phs1}. However, Proposition \ref{main result Phs2} concerning mild solutions can be applied only if $\alpha(t)\equiv\alpha>0$ is constant.
\begin{proposition}\label{solvability main results vibrating string} Under the conditions on the physical parameters $T, \alpha, \rho, k$ listed above we have:
	
\begin{enumerate}
	\item The abstract linear system associated with the controlled vibrating string \eqref{Eq1: State eq}, \eqref{Eq1: BCd} with output \eqref{output Vibrating} yields a non-autonomous boundary control and observation system on $(L^2([a,b];\K^2),\K^1,\K^2)$ if $u_2=0$, i.e., when the string is clamped at the end point $a,$ and in $(L^2([a,b];\K^2),\K^2,\K^2)$ if $u_2\neq 0.$ 
	\item  Let $\omega_0, \omega_1 \in H^1(a,b;\K)$ be such that $k\omega_0(b)+\omega_1(b)=u_1(0)$ and $\omega_0(a)=u_2(0)$.
	Then  \eqref{Eq1: State eq}-\eqref{Eq1: BCd} with output equation \eqref{output Vibrating} and initial conditions
	\[
\alpha(0)\rho(0,\cdot) \frac{\partial w}{\partial t}(0,\cdot) = \omega_0,
	\quad
	\frac{\partial \omega}{\partial t}(0,\cdot) = \omega_1
	\]
	has a unique solution $(\omega,y)$ such that $y\in C([0,\infty);\K^2)$ and \[t\longmapsto \begin{bmatrix}\alpha(t)\frac{\partial w}{\partial t}\\
	T(t,\cdot) \frac{\partial w(t,\cdot)}{\partial \zeta}\end{bmatrix}\in C^1\big((0,\infty); L^2(a,b;\K^2)\big)\cap C\big([0,\infty); L^2(a,b;\K^2)\big). \]
	\item Let $u_2\neq 0.$  Let $R(t), J(t)$ be self adjoint $2\times 2$-matrices such that $R\in L_{Loc}^\infty([0,\infty);\L(\K^2))$ and $c_0^{-1}\leq J(t)\leq c_0$ for all $t\geq 0$ and some constant $c_0>0.$ Choose $V, S$ in \eqref{output Vibrating} such that \eqref{Scattering passive-CondPHS} holds for all $t\geq 0.$ Then the linear system associated with the non-autonomous controlled vibrating string \eqref{Eq1: State eq}-\eqref{Eq1: BCd} and \eqref{output Vibrating} is a well-posed non-autonomous boundary control and observation system.
	\item Assume that $\alpha(t)\equiv\alpha>0$ is constant such that the assumptions in $(3)$ hold. Let $\omega_0, \omega_1 \in L^2(a,b;\C).$ Then \eqref{Eq1: State eq}-\eqref{Eq1: BCd} with  initial conditions
	\[
	\alpha(0)\rho(0,\cdot) \frac{\partial w}{\partial t}(0,\cdot) = \omega_0,
	\quad
	\frac{\partial \omega}{\partial t}(0,\cdot) = \omega_1
	\]
	has a unique (mild) solution $\omega$ such that  \[t\longmapsto \begin{bmatrix}\alpha(t)\frac{\partial w}{\partial t}\\
	T(t,\cdot) \frac{\partial w(t,\cdot)}{\partial \zeta}\end{bmatrix}\in  C\big([0,\infty); L^2(a,b;\K^2)\big). \]
\end{enumerate}	

	\end{proposition}


\par\noindent  
\subsection{Timoschenko beam}
Consider the following model of the Timoshenko beam with time-dependent coefficient and time dependent boundary control
\begin{align}\label{TimoshenkoEq1}
\frac{\partial }{\partial t}\big(\tilde\rho(t)\rho(t,\zeta)\frac{\partial w}{\partial t}(t,\zeta)\big)&=\frac{1}{\tilde\rho(t)}\frac{\partial}{\partial \zeta}\Big[ K(t,\zeta)\Big(\frac{\partial}{\partial \zeta}w(t,\zeta)+\phi(t,\zeta)\Big)\Big]
\\
\label{TimoshenkoEq2}\frac{\partial }{\partial t}\big(\tilde I_\rho(t)I_\rho(t,\zeta)\frac{\partial \phi}{\partial t}(t,\zeta)\big)&=\frac{1}{\tilde I_\rho(t)}\frac{\partial}{\partial \zeta}\Big(EI(t,\zeta)\frac{\partial ^2}{\partial \zeta}\phi(t,\zeta)\Big)+\frac{1}{\tilde \rho(t)}K(t,\zeta)\Big(\frac{\partial }{\partial \zeta}w(t,\zeta)-\phi(t,\zeta)\Big)
\end{align}
\begin{align}\frac{\partial w}{\partial t}(t,a)&=u_1,\qquad &t\geq 0
\\\frac{\partial\phi}{\partial t}(t,a)&=u_2,\qquad  &t\geq 0
\\ K(t,b)\Big[\frac{\partial w}{\partial \zeta}(t,b)-\phi(t,b)\Big]+\alpha_1\tilde\rho(t)\frac{\partial w}{\partial t}(t,b)&=u_3, \qquad  &t\geq 0
\\ EI(t,b)\frac{\partial \phi}{\partial \zeta}(t,b)+\alpha_2\tilde I_\rho(t)\frac{\partial\phi}{\partial t}(t,b)&=u_4,\qquad  &t\geq 0
\end{align}
for some positive constants $\alpha_1, \alpha_2\geq 0$. 
Here $\zeta \in (a,b)$, $t\geq 0$, $w(t,\zeta)$ is the transverse displacement of the beam and $\phi(t,\zeta)$ is the rotation angle of the filament of the beam. We assume that $K$, $\rho$, $EI$, $I_\rho \in C^2([0,\infty);L^\infty(a,b)) \cap C_b([0,\infty);L^\infty(a,b))$ and  there exists $m>0$ such that for a.e\  $\zeta\in[a,b]$ and all $t\geq 0$ we have \[m^{-1} \leq \rho(t,\zeta), K(t,\zeta),EI,I_\rho\leq m,\] where $\rho(t,\zeta)$ and $I_\rho$ are strictly positive. Moreover, $\tilde\rho, \tilde I_\rho \in C^1([0,\infty))$  are strictly positive. 

\par\noindent Indeed, taking as state variable $x:=(\frac{\partial w}{\partial \zeta}-\phi, \tilde \rho\rho\frac{\partial w}{\partial t}, \frac{\partial \phi}{\partial \zeta},\tilde I_\rho I_\rho\frac{\partial\phi}{\partial t})$ one can easily see that \eqref{TimoshenkoEq1}-\eqref{TimoshenkoEq2} can be written as a system of the form \eqref{port-Hamiltonian system}-with 
\medskip

$P_1=\begin{bmatrix}
0 &\tilde\rho^{-1}& 0 & 0\\
\tilde\rho^{-1} & 0 & 0 & 0\\
0 & 0 & 0 & \tilde I_\rho^{-1}\\
0 & 0 & \tilde I_\rho^{-1} & 0\\
\end{bmatrix} $, $P_0=\begin{bmatrix}
0 & 0 & 0 & -\tilde I_\rho^{-1}\\
0 & 0 & 0 & 0\\
0 & 0 & 0 & 0\\
\tilde I_\rho^{-1} & 0 & 0 & 0\\
\end{bmatrix}$ and $\H=\begin{bmatrix}
K & 0 & 0 & 0\\
	0 & \rho^{-1} & 0 & 0\\
	0 & 0 & EI & 0\\
0 & 0 & 0 & I_\rho^{-1}\\
\end{bmatrix}$.
\medskip

The boundary condition can be formulated as follows 
\begin{align*}\begin{bmatrix} 
0\\
0\\
0\\
0
\end{bmatrix}&=\begin{bmatrix} 0&0&0&0&0&1&0&0
\\ 0&0&0&0&0&0&0&1
\\ 1&\alpha_1&0&0&0&0&0&0
\\ 0&0&1&\alpha_2&0&0&0&0
\end{bmatrix}\begin{bmatrix} \H(t,b)x\\ \H(t,a)x\end{bmatrix} =: W_B \begin{bmatrix} \H(t,b)x\\ \H(t,a)x\end{bmatrix}
\end{align*}
Thus $W_B$ has full rank and the corresponding $4\times 8$ matrix $W_B(t)$ is given by 
\begin{align*}W_B(t)&= W_B \begin{bmatrix}
P_1^{-1}(t)& I \\
-P_1^{-1}(t)&I
\end{bmatrix}=\begin{bmatrix} -\tilde \rho(t)&0&0&0&0&1&0&0
\\ 0&0&-\tilde I_\rho(t)&0&0&0&0&1
\\\alpha_1\tilde \rho(t)&\tilde \rho(t)&0&0&1&\alpha_1&0&0
\\0&0&\alpha_2\tilde I_\rho(t)&\tilde I_\rho(t)&0&0&1&\alpha_2
\end{bmatrix}.
\end{align*}
Thus $W_B(t)\Sigma W_B^*(t)=\left[\begin{smallmatrix}0&0&0&0
\\ 0&0&0&0
\\ 0&0&4\alpha_1\tilde \rho(t)&0
\\ 0&0&0&\alpha_2\tilde I_\rho(t)
\end{smallmatrix}\right]\geq 0$. 
As in Example \ref{Ex vibrating string}, the output equation can be choosing similarly as \eqref{output Vibrating}.  Thus the above Timoshenko beam fit into the framework of port-Hamiltonian system and thus one obtain a similar results to that presented in Proposition \ref{solvability main results vibrating string}. 
\par 

\end{document}